\documentclass[draft]{elsarticle}

\usepackage{amsmath}
\usepackage{amssymb}
\usepackage{amsbsy}
\usepackage{mathrsfs}
\usepackage{graphicx}
\usepackage{pgfplots}
\pgfplotsset{compat=1.14}
\usepackage{tikz}
\usepackage[mathcal]{euscript}
\usepackage[normalem]{ulem}

\usepackage{cancel}

  \usepackage{algorithmic}
  \usepackage{algorithm}

\newtheorem{theorem}{Theorem}
\newtheorem{corollary}{Corollary}
\newtheorem{lemma}{Lemma}

\newdefinition{hypothesis}{Hypothesis}
\newdefinition{assumption}{Assumption}
\newdefinition{remark}{Remark}

\newproof{proof}{Proof}

\numberwithin{theorem}{section}
\numberwithin{corollary}{section}
\numberwithin{lemma}{section}
\numberwithin{hypothesis}{section}
\numberwithin{assumption}{section}
\numberwithin{remark}{section}

\usepackage{color}
\usepackage[colorinlistoftodos,textwidth=3.5cm]{todonotes}

\def\Vec#1{{\boldsymbol {#1}}}
\def\Du{\boldsymbol D(\Vec{u})}

\def\HVolund#1{H^{#1}(\Omega)}
\def\HzVoltroisd{\boldsymbol{H^1_0}(\Omega)}
\def\HVoltroisd#1{\boldsymbol{H^{#1}}(\Omega)}
\def\Lund#1{L^{#1}(\Omega)}

\def\vzero{{\boldsymbol  0}}
\def\phi {\varphi}
\def\barro{\bar{\rho}}
\def\barmu{\bar{\mu}}
\def\borneBeta{9/8}

\def\Com{C}
\def\Comi#1{C_{#1}}
\def\Comou{\mathscr{C}}
\def\Kom{K}
\def\Komix{\mathcal{K}}
\def\Komou{\mathscr{K}}

\def\Cgfi{\Comou^{\scriptscriptstyle {\nabla\phi}}}
\def\Cfid{\Komix^{\scriptscriptstyle{\phi}}_2}

\def\Cu{\Comou^{\scriptscriptstyle {\Vec{u}}}}
\def\Cudc{K^u}

\def\Cudpasf{\Komix^u \left( \frac{\dmu }{\mu_I}\right)^\frac{6}{d}   
\left(\left(\frac{\Cpc}  { \mu_I}\right)^\frac{6-d}{d}
 + \Komou^u \right)}

\def\Cud{\Komix^u \biggl(\frac{\dmu }{\mu_I}\biggr)^\frac{6}{d}   
\biggl(\biggl(\frac{\Cpc}{\mu_I}\biggr)^\frac{6-d}{d}
 + \Komou^u \biggr) + \Komou^f}

\def\Cgu{\Komou^{\scriptscriptstyle {\nabla\Vec{u}}}}
\def\Cugfi{\Comou^{\scriptscriptstyle {\Vec{u}\nabla\phi}}}

\def\Cpc{K^p}
\def\Cp{\Komix^p+\mu_SC\Cgu}

\def\dfi{\delta\phi}
\def\du{\delta\Vec{u}}
\def\dup{\delta\Vec{u}^p}

\def\dro{|\delta\rho|}
\def\dmu{|\delta\mu|}

\def\prodD#1{\left<#1\right>}
\def\prodL2#1{\left(#1\right)}

\def\SPCSHY{\textit{SPKCSHY--sequence}}

\begin{document}


\begin{frontmatter}
\title{Analysis of a time-discrete scheme for the {N}avier-{S}tokes/{A}llen--{C}ahn model\tnoteref{CRSNG}} 

\author[GIREF]{J.Deteix}
\ead{jean.deteix@mat.ulaval.ca}
\author[GIREF]{G.L. Ndetchoua Kouamo\corref{cor1}}
\ead{gerard-lionel.ndetchoua-kouamo.1@ulaval.ca}
\author[GIREF]{D. ̃Yakoubi}
\ead{yakoubi@giref.ulaval.ca}

\address[GIREF]{Groupe Interdisciplinaire de Recherche en \'El\'ements Finis de l'Universit\'e Laval, D\'epartement de Math\'ematiques et Statistiques, Universit\'e Laval,Qu\'ebec,Canada}

\cortext[cor1]{Corresponding author}
\tnotetext[CRSNG]{This work was supported by the Natural Sciences and Engineering Research Council of Canada (NSERC).}


\begin{abstract}
 This paper address the approximation of the dynamic of two fluids with non matching densities and viscosities modeled by the Allen-Cahn equation coupled with the time dependent  Navier-Stokes equations.
 Existence, uniqueness and a maximum principle are obtained for a totally implicit semi-discrete in time formulation. These results are based on an original stabilized 
 fixed point algorithm for which well posedness and convergence is analyzed. 
Numerical experiments are performed to show the influence of the iterative process. 
\end{abstract}


\end{frontmatter}

\section{Introduction}\label{Intro}

In the broad spectrum of approaches to model multi-fluid (multi-phase) flow and the capture of its interfacial behaviour, the phase-field approach is quite attractive as it is a physically motivated model based on the competition between the different species/phases.  We refer the readers to~\cite{GroReu2007, GroReu2011, BotReu2017} for a review of some of the most common approaches (both theoretically and numerically).

This work is part of an effort to produce a new numerical algorithm for the approximation of the complex behaviour of a binary mixture of fluids.
More precisely, we are interested in the use of the Allen-Cahn model for the description of the dynamic of the mixture of two non miscible fluids coupled to the unsteady Navier-Stokes equations describing fluids displacements. This leads to what is known as an unsteady  \textit{Navier-Stokes-Allen-Cahn} (NS-AC) model.

We consider an unsteady flow of two immiscible incompressible fluids of nonnegative constants  densities  $\rho_a, \rho_b >0$ and viscosities $\nu_a,\nu_b>0$. 
Introducing the scalar function $\phi$ (called a phase field)
$$
\phi(t,\Vec{x}) = \begin{cases}
-1 & \Vec{x} \text{ in fluid a}\\
1 & \Vec{x} \text{ in fluid b}\\
\end{cases}
$$
representing the difference of volume fraction of both fluids, we define
the density and viscosity of the fluids mixture as
\begin{equation}\label{EqRhoMu0}
\rho(\phi) = \barro + \frac{\delta\rho}{2} \phi,
\qquad
\mu(\phi) = \barmu+\frac{\delta\mu}{2} \phi
\end{equation}
with
\begin{equation}\label{BorneRhoMu0}
\begin{array}{c}
\displaystyle \bar{\rho} = \frac{\rho_a+\rho_b}{2},\ \  \delta\rho = \rho_a-\rho_b,\ \  \rho_I = \min\{\rho_a, \rho_b\},\ \   \rho_S = \max\{\rho_a,\rho_b\}. 
\vspace*{8pt}
\\
\displaystyle \bar{\mu} = \frac{\mu_a+\mu_b}{2},\ \  \delta\mu = \mu_a-\mu_b,\ \   \mu_I = \min\{\mu_a, \mu_b\},\ \  \mu_S = \max\{\mu_a,\mu_b\}.
\end{array}
\end{equation}
The behavior of $\phi$ is characterized using the Allen-Cahn model which describes the  separation of a mixture composed of two species (see \cite{Emm2003, BotReu2017, GroReu2011} concerning the physical details of the model). 
Although this model raises questions (mass conservation is violated), we neglect modifications to the system in order to recuperate the mass conservation as discussed in \cite{CheHilLog2010,AlfAli2014} for instance. Then the phase field satisfy 
\begin{equation}\label{EqAC}
\partial_t \phi + \left(\Vec{u} \cdot \nabla \right) \phi = 
\gamma\left( \Delta \phi - \frac{d}{d\phi}F(\phi) \right)
\end{equation}
where the nonnegative constant $\gamma>0$ describes the mobility coefficient, $F$ is the free energy density of the mixture, (a double well potential), defined as 
\begin{equation}\label{Appr-F}
F(\phi) = \frac{1}{4\eta^2}(\phi^2 -1)^2\qquad\displaystyle \frac{dF(\phi)}{d\phi} = f(\phi) = \frac{1}{\eta^2} \phi
\left( \phi^2 -1\right)
\end{equation}
and $0<\eta$ is a parameter related  to the thickness of the transition zone between both fluids. Assuming this thickness is small with respect to the other dimensions, in fact we assume that $\eta^2\ll\gamma$, the interface between both fluids can be described as the zero level of $\phi$.

The incompressible nature of the fluids imposes a divergence free velocity $\Vec{u}$, and following \cite{SheYan2010} (for example) the momentum equation of the mixture is 
\begin{equation}\label{EqMoment}
\begin{aligned}
&\displaystyle \sqrt{\rho} \partial_t \left( \sqrt{\rho} \Vec{u}\right) 
+ \rho \left( \Vec{u} \cdot \nabla\right) \Vec{u} 
+ \displaystyle\frac{1}{2} \nabla \cdot \left( \rho \Vec{u} \right) \Vec{u} - \nabla \cdot  \left( \mu \Du\right) + \nabla p 
\\
& \qquad\qquad\qquad\qquad\qquad\qquad\qquad\qquad\qquad\qquad
= - \sigma  \nabla \cdot \Vec{T}(\phi) + \Vec{G}(\phi)
\end{aligned}
\end{equation}
where $p$ is the fluids pressure  and $\Du $ is  the shear rate  defined by
$$
\Du = \frac{1}{2}(\nabla\Vec{u} + \nabla\Vec{u}^t).
$$
As for the forces involved, $\Vec{G}$  represents external volumetric force (such as gravity), and the extraneous elastic stress induced by the surface tension is defined as
\begin{equation*}
\Vec{T}(\phi) = \sigma\nabla \phi \otimes \nabla \phi,
\end{equation*}
where $\sigma$ represents the density of the mixing energy. 
Before going any further, based on the 
potential of the free energy~\eqref{Appr-F} and~\eqref{EqAC}
we can replace the surface tension term 
\begin{equation*}
\begin{aligned}
\nabla \cdot \Vec{T}(\phi) =  
\sigma\nabla \cdot \left( \nabla\phi \otimes \nabla\phi \right)  
&=\sigma\left( \Delta \phi - f(\phi) \right) \nabla \phi + \sigma\nabla \left(\frac{1}{2}|\nabla \phi|^2 + F(\phi)\right) \\ 
&= \frac{\sigma}{\gamma}\left( \partial\phi_t  + \Vec{u} \cdot \nabla \phi  \right) \nabla \phi - \nabla \hat p,
\end{aligned}
\end{equation*}
where $\hat p = \displaystyle \sigma(\frac{1}{2}|\nabla \phi|^2 + F(\phi))$ is of the same nature as a pressure. As in \cite{SheYan2010}, we introduce a new pressure, $p + \hat p$, still denoted $p$. Then we can rewrite the momentum \eqref{EqMoment},
and the system describing the behaviour of the mixture is 
\begin{equation}\label{EqNS0}
\left\{
\begin{aligned}
&\displaystyle \sqrt{\rho} \partial_t \left( \sqrt{\rho} \Vec{u}\right) 
+ \rho \left( \Vec{u} \cdot \nabla\right) \Vec{u} 
+ \displaystyle\frac{1}{2} \nabla \cdot \left( \rho \Vec{u} \right) \Vec{u} - \nabla \cdot  \left( \mu \Du\right) + \nabla p 
\\
& \quad\quad\qquad\qquad\qquad\qquad\qquad\qquad
=
- \frac{\sigma}{\gamma}\left( \partial_t\phi  + \Vec{u} \cdot \nabla \phi  \right) \nabla \phi
+ \Vec{G}(\phi),
\\
&\nabla\cdot\Vec{u} = 0,
\\
&\partial_t \phi + \left(\Vec{u} \cdot \nabla \right) \phi = 
\gamma\left( \Delta \phi - \frac{1}{\eta^2} \phi
\left( \phi^2 -1\right) \right).  
\end{aligned}
\right.
\end{equation}

As this system is nonlinear and strongly coupled, the numerical approximation of solutions is challenging. 
Existence of solutions for this system is known (see~\cite{LiDinHua2016} and the references therein), and semi-explicit numerical strategies  are documented (see, for example, the approach proposed in~\cite{SheYan2010}). 
Such explicit schemes, if they are generally perceived as efficient (at least from a computational standpoint), frequently lacks in accuracy and require important constraint on the time step to recover the loss in precision.  

From empirical observations, for a given precision, implicit scheme allows larger time step when compared to explicit methods. This makes implicit methods attractive as they frequently lead to less computational effort for a fixed time horizon (Table~\ref{Le_tablo} for example).  
Aiming for a more accurate while flexible approach, this paper addresses some theoretical aspects related to the numerical approximation of \eqref{EqNS0} using implicit time discretization.

Based on an original fixed point (the iterates are defined by~\eqref{AlgAC}--\eqref{AlgNSDivu}), this paper presents 
the existence and uniqueness of the solution for a semi-discrete formulation of \eqref{EqNS0}. The result is based on the study of the well-posedness (Theorem~\ref{Th-existnce}), uniform boundness (Theorems~\ref{TheoBorneUPhi}) and the strong convergence (Theorem~\ref{convergenceH2}) of the fixed point problem.
Finally, as the fixed point problem produces bounded phase fields, a maximum principle is obtained for the limit problem \eqref{EqSysSD}. 

%
\section{Preliminaries}\label{prelim}
\subsection{Spaces}
In what follows, $\Omega$  is a bounded open convex of class $C^{2,1}$ of $\mathbb R^d,~ d=$  2 or 3, and $\partial \Omega$ denotes the boundary of $\Omega$.
$L^p(\Omega)$ represent the usual set of $p-th$ power measurable functions, and $\left(L^p(\Omega)\right)^d = \Vec{L}^p(\Omega)$. The scalar product defined on $L^2(\Omega)$ or $\Vec{L}^2(\Omega)$ is denoted (without distinction) by  $\prodL2{\cdot,\cdot}$ and its norm $\|\cdot\|$. The space $L^2_0(\Omega)$ is defined as
$$
L^2_0(\Omega) =\left \{ q \in L^2(\Omega);\ \displaystyle \int_{\Omega} q(\Vec{x}) d\Vec{x}  = 0  \right \}. $$

The Sobolev spaces, denoted $W^{m,p}(\Omega)$ (we denote $\Vec{W}^{m,p}(\Omega)= \left(W^{m,p}(\Omega)\right)^d$), with $p\in [1,+\infty)$ and $m$ integer, is defined as 
$$
W^{m,p}(\Omega) = \left\{u\in L^p(\Omega)\ : D^{\alpha} u \in L^p(\Omega)\ \forall |\alpha|\le m\right\}
$$
where $\alpha$ is a multi-index in $\mathbb{N}^d$. These spaces are equipped with the norm $\|\cdot\|_{m,p}$ and semi-norm $|\cdot|_{m,p}$.

The spaces  $W^{s,2}(\Omega)$ and $\Vec{W}^{s,2}(\Omega)$, $s\in \mathbb{R}$, are denoted $H^s(\Omega)$ and $\Vec{H}^s(\Omega)$ respectively. Their norm are denoted $\|\cdot\|_s$ and semi-norm $|\cdot|_s$. Without distinction for the dimension, we denote the duality pairing between $H^1_0(\Omega)$ and its dual $H^{-1}(\Omega)$ (or between $\Vec{H}^1_0(\Omega)$ and $\Vec{H}^{-1}(\Omega)$)) by $\prodD{\cdot, \cdot}$. More generally, for a space $V$ and $V'$ its dual, we denote the duality pairing by $\prodD{\cdot, \cdot}_{V',V}$.



For a fixed positive real number $T$ (representing the final time) and a  separable Banach space $E$ equipped with the norm $\|\cdot\|_{E}$,  we denote by $\mathcal{C}^0(0,T;E)$ the space of continuous functions from $[0,T]$ with values in $E$. For a positive integer $p$, we introduce the Bochner spaces, 
$$
L^p(0,T;E) = \left\{
u: (0,T) \mapsto E :\ \left(\int_0^T \|u(\tau)\|_E^p\,d\tau\right)^{1/p} < \infty\right\}
$$
We refer to \cite{Ada2003}, \cite[Chapter 2]{BoyFab2013} for details concerning those spaces.
Let us remind the  following interpolation  inequality for Sobolev spaces (see \cite[Exercise II.3.12]{Gal1994} for instance): 

\begin{lemma} Assume $\Omega$ locally Lipschitz and let 
$$r  \in \left[q, \frac{qd}{d-q}\right], \;\; \mbox{if} \; \;q \in [1,d), \qquad 
\mbox{and}\qquad r  \in [q, +\infty) \;\; \mbox{if} \;\; q \ge d.
$$
The following inequality holds  for all $u \in W^{1,q}(\Omega)$
\begin{equation}\label{Galdi}
\| u \|_{L^r} \le {\Com} \| u\|^{1-s}_{L^q} \, \|u \|^{s}_{1,q} \qquad \mbox{where}\qquad 
s = \frac{d(r-q)}{r q}.
\end{equation}
\end{lemma}
The assumption on $\Omega$ are sufficient to verify the inf–-sup or Ladyzhenskaya-–Babu\u{s}ka–-Brezzi (LBB) condition (see~\cite{Bab1972,BofBreFor2013, BoyFab2013}), that is, there exists a positive constant $\kappa>0$ depending only on $\Omega$, 
such that 
$$
- \inf_{q \in L_0^{2}(\Omega)} \sup_{\Vec{v} \in \HzVoltroisd}  
\frac{\left( q, \nabla \cdot \Vec{v}  \right)}{ \|q \|\| \Vec{v} \|_{1}} \geq \kappa. \label{LBB}
$$

Most proofs will rely on combination of Sobolev embedding, Poincar\'e, Korn, Poincar\'e-Wirtinger, Cauchy-Schwarz, Holder and Young inequalities (see~\cite{Ada2003, Bre1974} for details on these concept). We will denote by \textbf{\SPCSHY} such combination.

\subsection{Equalities and constants}\label{sec_Constante}
In order to simplify the proofs, we introduce the following equalities, for arbitrary fields  $\Vec{w}_0,\Vec{w}_1\in\mathbb{R}^d$ and $\phi_0,\phi_1\in\mathbb{R}$ 
\begin{equation} \label{Cauchy_laplace}
\begin{split}
&\int_\Omega  ( \mu(\phi_0) \, \Vec{D}(\Vec{w}_0)   
-\, \mu(\phi_1) \, \Vec{D}(\Vec{w}_1) )
: \Vec{D} \left(\Vec{w}_0 - \Vec{w}_1 \right)
=
\\ &\qquad
\int_\Omega  \mu(\phi_0) \left| \Vec{D} \left(\Vec{w}_0 - \Vec{w}_1 \right) \right|^2
+ \frac{\delta\mu}{2}\int_\Omega  \left( \phi_0 -  \phi_1 \right)  \, \Vec{D}  (\Vec{w}_1)  :
\Vec{D}  \left(\Vec{w}_0 - \Vec{w}_1 \right).
\end{split}
\end{equation}
%
%
If $\nabla\cdot \Vec{w}_0 =  \nabla\cdot \Vec{w}_1 =0 $ then 
\begin{equation}\label{rhoUVV}
\begin{split}
\int_\Omega \rho(\phi_0)(\Vec{w}_0 \cdot \nabla)\Vec{w}_1\cdot\Vec{w}_1
=  
-\frac{1}{2}\int_\Omega \nabla\cdot(\rho(\phi_0)\Vec{w}_0)\Vec{w}_1\cdot\Vec{w}_1
\end{split}
\end{equation}
\begin{equation}\label{UgradPhiPhi}
\begin{split}
\int_\Omega (\Vec{w}_0 \cdot \nabla \phi_{0})\phi_{1}
=  
-\int_\Omega (\Vec{w}_0\cdot \nabla\phi_{1})\phi_0, \qquad \int_\Omega (\Vec{w}_0 \cdot \nabla \phi_{0})\phi_{0}=0
\end{split}
\end{equation}
\begin{equation}\label{UgradPhiPhi1}
\begin{split}
\int_\Omega ( \Vec{w}_0 \cdot \nabla \phi_{0} -  \Vec{w}_1 \cdot \nabla \phi_{1} ) \left(\phi_{0}  - \phi_{1}\right)
&=  
\int_\Omega \left(\Vec{w}_0 - \Vec{w}_1 \right) \cdot \nabla \phi_{1}
\left(\phi_{0}  - \phi_{1}\right)
\\
&=  
-\int_\Omega \left(\Vec{w}_0 - \Vec{w}_1 \right)\cdot
\nabla \left(\phi_{0}  - \phi_{1}\right)   \phi_{1}.
\end{split}
\end{equation}
Some algebraic identities frequently used in this paper, $\forall\, a,b,c,d\in \mathbb{R}$
\begin{equation} \label{Cauchy_convect}\begin{aligned}
a^2 b - c^2 d &= a^2 (b-d)+d(a+b)(a-b) \\ ab-cd &= a(b-d)+(a-c)d = (a-c)b+(b-d)c.
\end{aligned}
\end{equation}
In this work, the dependency of constants with respect to physical parameters plays an important role, in particular the behaviour of constants in relation with the parameters $\mu_I$ and $\mu_S$ defining $\mu$ the viscosity. $\Com$ will denote generic constants depending \textit{at most} on $\Omega$ and its boundary (for example, the Poincar\'e and the Korn  constant). $\Kom$ will denote generic constants depending on any physical data. $\Comou$ will denote generic constants that could depend on the physical parameter and state variables \textit{with the exception} of $\mu$, $\mu_I$ and $\mu_S$. $\Komou$ will denote constant depending on all the physical values with the addition that
$$\lim\limits_{\mu_I\to\infty}\Komou = 0$$
typically $\Komou= \mathscr{C}\mu_I^{-\alpha}$ with $\alpha >0$.
Finally $\Komix$ will denote special constant of the form
\begin{equation}\label{Def_K}
\Komix = \mathscr{C} + \mathscr{K}.
\end{equation}
Lastly, when there is no possible confusion, we will make no distinction between the various values of $\Com$, $\Comou$, $\Komou$, $\Komix$ and $K$ when manipulating expressions.
\subsection{Existence and regularity results}
Essential element of this work, we recall a 
lemma for elliptic problem with Neumann boundary conditions 
(for instance \cite[Chap. 3]{Gri1985}  or~\cite{LioMag1968, Pao1992,QuaVal1999}). 
Once again we underline that the assumptions previously made on $\Omega$ and its boundary makes it possible to apply such regularity lemma. 
\begin{lemma}\label{Lemma_Elliptique}
Let $a(\cdot), \Vec{b}(\cdot)$ and $c(\cdot)$ be three functions such that $a,c \in  L^\infty(\Omega) $ and $\Vec{b} \in \mathbf{H^1_0}(\Omega)$. Assume 
$a(\Vec{x}) \ge a_0 > 0, \, c(\Vec{x})\ge c_0 >0$, $\nabla\cdot \Vec{b} = 0$ and $f\in \HVolund{r}, \, r \ge 0$ . Then the following problem 
\begin{equation}
\left\{
\begin{array}{rcl}
-a\Delta\phi + \Vec{b}\cdot\nabla\phi  + c\phi&=& f  \qquad \mbox{in} \qquad \Omega \\ 
\partial_\Vec{n} \phi &=& 0 \qquad \mbox{on} \qquad \partial \Omega
\end{array}
\right.
\end{equation}
admits a unique solution which satisfies 
\begin{align}
\| \phi \|_{2+r} & \le C(\Omega,a_0,c_0,\Vec{b},r) \left( \| f \|_{r} + \| \phi\|_{1+r}  \right). \label{Est-Ellip}
\end{align}
where $C(\Omega,a_0,c_0,\Vec{b},r)$ is a constant depending on $\Omega, a_0,\, \Vec{b}, c_0$ and $r$.
\end{lemma}

Of course, \eqref{EqNS0}
is a well defined system for 
$(\Vec{u}, p,\phi)$ once it is completed with boundary and initial conditions. For the sake of simplicity, we  consider homogeneous Dirichlet boundary
conditions for $\Vec{u}$ and homogeneous Neumann condition on $\phi$
\begin{alignat}{2}
\left( \Vec{u}, \partial_\Vec{n} \phi \right) &= \left( \vzero, 0\right) 
\qquad &\text{on }\,  \partial \Omega   \label{BC} \\ 
\left( \Vec{u}, \phi \right)_{t=0} &= \left( \Vec{u}^0, \phi^0\right),\quad\|\phi^0\|_{L^{\infty}} = 1 
&\text{in }\,    \Omega   \label{IC}.
\end{alignat}
%
%

The usual hypothesis leading to a Boussinesq-type approximation (such as having a ratio of densities near 1) have no effects on the results presented here. Therefore a more general case is considered with the following assumptions.
\begin{assumption}\label{HypoG}
The function $ \Vec{G}:\; \mathbb{R} \rightarrow \mathbb{R}^d $\; is a $ C^1(\mathbb{R})$ function such that  
\begin{equation*}
\Vec{G}(0) = \vzero, \qquad \| \Vec{G}'\|_{L^\infty} \le 1 \qquad \mbox{and} \quad \forall s \in \mathbb{R},\qquad |\Vec{G}(s)| \le s. 
\end{equation*}
\end{assumption}
%
%
%
The existence of a solution to~\eqref{EqNS0}, \eqref{BC}--\eqref{IC} is proved provided basic compatibility conditions are satisfied, namely 
\begin{assumption}\label{ReguInitial} 
The data $\left(\Vec{u}^0,\phi^0\right)$ satisfies the regularity conditions 
\begin{equation*}
\Vec{u}^0 \in \HzVoltroisd \cap \HVoltroisd{2} \; \mbox{and} \; 
\phi^0 \in  \HVolund{3}
\end{equation*}
and the compatibility condition 
\begin{align*}
-\nabla \left(\mu_0 \Vec{D}(\Vec{u}^0)\right) + \nabla p_0 + \sigma \nabla \cdot \Vec{T}(\phi^0) = \Vec{G}(\phi^0), \label{Comp-Cond} 
\end{align*}
for some $p_0 \in H^1(\Omega)$.
\end{assumption}
Finally, we have an existence and uniqueness result for the continuous problem 
thanks to \cite{JiaTan2009,ZhaGuoHua2011,LiDinHua2016}:
\begin{theorem}\label{Existence}
Under Assumptions~\ref{HypoG}--\ref{ReguInitial} 
there exists  $T^*>0$  and a unique  solution $\left(\Vec{u},p,\phi\right)$ to the problem 
\eqref{EqNS0}, \eqref{BC}--\eqref{IC} such that
\begin{equation}\label{Reg-Sol}
\begin{aligned}
& \Vec{u} \in L^{\infty}\left(0,T^*; \HzVoltroisd \cap \HVoltroisd{2} \right)  
\cap  L^2\left(0,T^*;\Vec{W}^{2,q}(\Omega) \right),
\; \nabla \cdot \Vec{u}|_\Omega =0 
\\
& p \in L^\infty \left( 0,T^*; H^1(\Omega)\right)  \cap  L^2\left(0,T^*;W^{1,q}(\Omega) \right),
\\
&\phi \in L^{\infty}\left(0,T^*; H^3(\Omega) \right)  \cap L^2\left(0,T^*;H^{4}(\Omega) \right)
\\
\end{aligned}
\end{equation}
for some $q, \; d < q <  \displaystyle \frac{2d}{d-2}$.
\end{theorem}
The rest of this work relies on the regularity of $\phi$ and $\Vec{u}$ given by Theorem~\ref{Existence}. 
We emphasize that the regularity imposed by Assumption~\ref{ReguInitial}
is not excessive. The proofs of most of the results presented here hinges on these assumptions (and consequently on \eqref{Reg-Sol}). Nevertheless, we do not exclude that results similar to those presented here could be obtained with less regular assumptions on the initial values.

\section{A fully implicit semi-discrete formulation}\label{Sec-Pb-Initial}

For the sake of simplicity and clarity, a first order approximation of the time derivative was chosen for both the Navier-Stokes and advection-diffusion equations. 
Different treatment of \eqref{EqNS0} could be considered, 
from a backward Euler (BDF1) to a forward Euler scheme. 
The resulting system contains two important difficulties: it is \textit{non linear} and \textit{strongly coupled} in $\Vec{u}$ and $\phi$.

Excluding the time derivative, the terms in~\eqref{EqNS0} can be treated: implicitly (
leading to various fixed-point algorithm), semi-explicitly 
(using known values of the state variables 
, see \cite{SheYan2010} for example) or totally explicitly (in which case no system has to be solved
). Of course it is only in the implicit case that we have a proper backward Euler scheme with, \textit{a priori},  all its properties.

We emphasizes that for all these strategies (implicit, semi-explicit and explicit), the resulting system  will still be a coupled system. 
Different approach can be used to deal with the coupling of $\Vec{u}$ and $\phi$, a \textit{strongly coupled approach} would consist in solving both equation as a system (possibly non linear). A \textit{weakly coupled approach} would consist in solving each equation 
once at each time step, expressing the coupling terms in some explicit way (for example by using a Richardson extrapolation).

Any weakly coupled strategy reduces the work load at each time step. However, as for explicit or semi-explicit time scheme, this leads inevitably to conditional stability and certainly imposes conditions on the time step length offsetting any numerical advantages.
%
A strongly coupled approach avoiding these inconvenient is preferable. Such strategies imply to solve a system in $(\Vec{u},\phi)$ using a fixed-point approach. An implicit approach for the non linear terms has been retained as well, leading to a \textit{fully implicit approach}. 

Using a uniform  time-step $\Delta t >0$ and denoting 
$$t^n = n\Delta t,\quad \Vec{u}^{n} = \Vec{u}(t_{n}, \Vec{x}),\quad p^n = p(t_n,\Vec{x}),\quad \phi^n = \phi(t_n,\Vec{x}),$$
the time discretization of  \eqref{EqNS0} 
result in a sequence of nonlinear coupled (strongly) problems of the form
\begin{equation}\label{EqSysSD}
\left\{
\begin{aligned}
&\displaystyle \frac{\phi^{n+1}-\phi^n}{\Delta t} + \Vec{u}^{n+1} \cdot \nabla \phi^{n+1} - \gamma\left(\Delta\phi^{n+1} - \frac{\phi^{n+1}}{\eta^2}((\phi^{n+1})^2 -1)\right) = 0
\\ &
\frac{\rho^{n+1}\Vec{u}^{n+1} - \sqrt{\rho^{n+1}\rho^{n}}\Vec{u}^{n}}{\Delta t}
+ \rho^{n+1} ( \Vec{u}^{n+1} \cdot \nabla) \Vec{u}^{n+1} 
\\&\quad
+ \displaystyle\frac{1}{2} \nabla \cdot \left( \rho^{n+1} \Vec{u}^{n+1} \right) \Vec{u}^{n+1} 
- \nabla \cdot (\mu^{n+1} \Vec{D}(\Vec{u}^{n+1})) + \nabla p^{n+1} 
\\ &\quad
+ \frac{\sigma}{\gamma}\left( \frac{\phi^{n+1}-\phi^n}{\Delta t}\right) \nabla \phi^{n+1}  
+ \frac{\sigma}{\gamma}(  \Vec{u}^{n+1} \cdot \nabla \phi^{n+1}) \nabla \phi^{n+1}
= \Vec{G}(\phi^{n+1}) 
\\
&\nabla \cdot \Vec{u}^{n+1} = 0 
\\
&\displaystyle\rho^{n+1} = \barro + \frac{\delta\rho}{2} \phi^{n+1},\qquad 
\mu^{n+1}  =  \barmu + \frac{\delta\mu}{2} \phi^{n+1},
\end{aligned}
\right.
\end{equation}
completed with the boundary conditions~\eqref{BC}  and 
$ \left(\Vec{u}^0, p^0,\phi^0\right) =  \left(\Vec{u}_0,p_0,\phi_0\right).$

\begin{remark}
Through simple modifications of the terms involved in the system \eqref{EqSysSD} different schemes could be considered. Moreover, limiting the fixed point loop needed to solve \eqref{EqSysSD} to one iteration at each time step would produce a weakly coupled scheme.
\end{remark}

To finish this section, let us present the main result of this work:
the well posed character of \eqref{EqSysSD} and the respect of a maximum principle. More precisely, in the next section we intend to demonstrate the following,
\begin{theorem}\label{Le_resultat}
Under the assumptions of Theorem~\ref{Existence}, with $\Vec{u}^0$ divergence free. 
For $\mu_I$ sufficiently large there exist a $\tau\in ]0,1[$ depending on $\mu_I$ such that for all $\Delta t\le \tau\eta^2/\gamma$ the system \eqref{EqSysSD}, \eqref{BC}--\eqref{IC} admits a unique solution
$(\Vec{u}^n, p^n, \phi^n)$ in $\HzVoltroisd \cap \HVoltroisd{2} \times L^2_0(\Omega)\cap H^1(\Omega) \times H^3(\Omega)$ 
with  $$\nabla\cdot\Vec{u}^n = 0,\qquad\|\phi^n\|_{L^\infty}\le 1$$
at each time step $t^n = n\Delta t$.
\end{theorem}
%
%
%
\section{Existence, uniqueness and a maximum principle.}

%
%
In order to prove Theorem~\ref{Le_resultat}, we introduce, at each time step, a linearized coupled fixed point sequence. 
We will then show the convergence of this sequence 
and, through identification of the limit as a solution of \eqref{EqSysSD}, we will get the result announced here.

\subsection{A coupled iterative scheme}
As we consider a fixed point loop, at each time step 
we introduce $\Vec{u}^{n+1}_k$ $p^{n+1}_k$ and $\phi^{n+1}_k$ the state variables at the $k-th$ iteration of the fixed point loop at time $t^{n+1}$. 
The index denoting the iteration number of the fixed point loop and the superscript the time step number. 

At  time $t^{n+1}$, knowing $\left(\Vec{u}^n, p^n, \phi^n \right) $, we consider the following initialization  
$$\phi^{n+1}_0 = \phi^n, \qquad \Vec{u}^{n+1}_0 = \Vec{u}^n \qquad  \mbox{and} \qquad p^{n+1}_0 = p^n.$$
In what follows, to simplify the notation, when there is no ambiguity, we will neglect the time step superscript on $\Vec{u}$, $p$ and $\phi$. 
For solving the nonlinear phase field equation 
\begin{equation}\label{EqACDisc}
\frac{\phi^{n+1}-\phi^n}{\Delta t} + \Vec{u}^{n+1} \cdot \nabla \phi^{n+1} - \gamma\Delta\phi^{n+1} + \gamma f(\phi^{n+1}) = 0
\end{equation}
two obvious choices are a Picard fixed point or a Newton-type method. Considering the robustness of the method and fact that $f$ is a third degree polynomial in $\phi^{n+1}$, the Newton approach is an appropriate  choice. However, from numerical experiments the Picard fixed point seems to have comparable performance to the Newton like fixed point (see the Numerical tests in section~\ref{NumRes}). 

Replacing $f(\phi_{k+1})$ 
by its first order development:
$$
f(\phi_{k+1}) \approx \frac{\phi_{k}}{\eta^2}(\phi_{k}^2 -1) + \frac{1}{\eta^2}(3\phi_k^2-1)(\phi_{k+1}-\phi_k) = \frac{1}{\eta^2}(3\phi_k^2-1)\phi_{k+1} -\frac{2}{\eta^2}\phi_k^3,
$$
we get for \eqref{EqACDisc}
$$
\displaystyle\frac{\phi_{k+1}-\phi^n}{\Delta t} + \Vec{u}_{k} \cdot \nabla \phi_{k+1} - \gamma\Delta\phi_{k+1} + \frac{\gamma}{\eta^2} (3\phi_k^2-1)\phi_{k+1} - \frac{2\gamma}{\eta^2}\phi_k^3= 0.
$$
The convergence analysis will rely heavily on the fact that the phase function satisfy, at each fixed point iteration
\begin{equation}\label{MaxPrinc}
\|\phi_{k+1}\|_{L^\infty(\Omega)} \le 1. 
\end{equation}
To insure such condition, we introduce the supplementary term
\begin{equation}\label{TermBeta}
\frac{\gamma}{\eta^2}\beta(\phi_{k+1}-\phi_k)
\end{equation}
where $\beta$ is an arbitrary nonnegative real constant. 
Obviously, at convergence this term will be  zero. We will show that, provided $\beta$ is above a specific lower bound, $\phi_k$ will satisfy  \eqref{MaxPrinc}  without any additional condition.

Using a simple linearization of the convective term in the momentum equation in \eqref{EqSysSD},
we get the following algorithm: 
at each time step $t^{n+1}$, \\

\noindent \textit{1. Initialization:} \; 
$\left(\Vec{u}_0,p_0,\phi_0\right) =\left(\Vec{u}^n,p^n,\phi^n\right). 
$\\

\noindent \textit{2. \underline{Until convergence}, knowing $\phi_k$ and $\Vec{u}_k$ compute:}
\begin{equation}\label{AlgAC}
\left \{
\begin{array}{l}
\begin{aligned}
\left(1+ \frac{\gamma \Delta t}{\eta^2} \left( \beta -1 + 3\phi_k^2 \right)\right)\phi_{k+1}& + \Delta t \Vec{u}_k \cdot \nabla \phi_{k+1}
-\gamma \Delta t \Delta \phi_{k+1} \\
& = \phi^n + \frac{\gamma \Delta t}{\eta^2} \phi_k\left(\beta  + 2\phi_k ^2\right) \quad \mbox{in}\, \Omega
\end{aligned} \\
\partial_{\Vec{n}} \phi_{k+1} = 0\quad \mbox{on}\quad \partial \Omega. 
\end{array}
\right.
\end{equation}
\begin{equation}\label{AlgRhoMu}
\rho_{k+1} = \barro + \frac{\delta\rho}{2} \phi_{k+1},\qquad 
\mu_{k+1}  =  \barmu + \frac{\delta\mu}{2} \phi_{k+1},
\end{equation}
\begin{equation}\label{AlgNSDivu}
\left \{
\begin{array}{l}
\begin{aligned}
\rho_{k+1} \frac{ \Vec{u}_{k+1}}{\Delta t}
&+\rho_{k+1}\left( \Vec{u}_{k}\cdot \nabla \right) \Vec{u}_{k+1}
+\frac{1}{2} \nabla \cdot \left(\rho_{k+1}\Vec{u}_{k}\right) \Vec{u}_{k+1}  \vspace*{4pt}\\
&- \nabla \cdot \left( \mu_{k+1} \Vec{D}(\Vec{u}_{k+1}) \right) + \nabla p_{k+1}\vspace*{4pt}\\
&+\frac{\sigma}{\gamma} \left( \Vec{u}_{k+1} \cdot \nabla \phi_{k+1} \right)  \nabla \phi_{k+1}
+ \frac{\sigma}{\gamma \Delta t}\left( \phi_{k+1} - \phi^n\right) \nabla \phi_{k+1} \vspace*{4pt}\\ 
&=
\Vec{G}(\phi_{k+1}) + \sqrt{\rho_{k+1}}\frac{\sqrt{\rho^{n}} \Vec{u}^{n}}{\Delta t} \quad \mbox{in}\, \Omega
\end{aligned} \vspace*{4pt}\\
\nabla \cdot \Vec{u}_{k+1} = 0  \quad \mbox{in}\, \Omega \vspace*{4pt}\\
\Vec{u}_{k+1} = \vzero \quad \mbox{on}\quad \partial \Omega. 
\end{array}
\right.
\end{equation}

\subsection{Analysis of the iterative scheme} 
The convergence of the fixed point \eqref{AlgAC}--\eqref{AlgNSDivu} to a solution of the time discrete system \eqref{EqSysSD}, will be demonstrated in four steps: 
existence of the solution, regularity,  convergence of the fixed point and finally identification of the limit as a solution of \eqref{EqSysSD}.

Although 
an existence result for 
\eqref{AlgAC}--\eqref{AlgNSDivu} is our main goal in this first step, 
as $\phi^n$ must be a phase field, the bound \eqref{MaxPrinc} must be established. 

\subsubsection{Well-posedness} 
\begin{theorem}\label{Th-existnce}
Assuming the hypothesis of Theorem~\ref{Existence} with $\Vec{u}^0$ divergence free.
For all step length $\Delta t>0$ and at each time step $n \ge 0$ and iteration $k\ge 0$ there exist a solution 
$\left(\Vec{u}_{k+1},p_{k+1},\phi_{k+1} \right)$ of \eqref{AlgAC}--\eqref{AlgNSDivu} in 
$ \HzVoltroisd \cap \HVoltroisd{2} \times L^2_0(\Omega)\cap H^1(\Omega) \times H^3(\Omega)$. 
Moreover 
\begin{equation}\label{CondPdM}
\|\phi_{k+1}\|_{L^\infty(\Omega)} \le 1,\qquad \forall\, \beta \ge \borneBeta.
\end{equation}
\end{theorem}
\begin{proof}
First we  prove that the result holds for $n=0$ and $k=0$ which is the first step in the iterative process for $t^1 = \Delta t$. 
Introducing
$$\alpha = \frac{\gamma \Delta t}{\eta^2},\quad g(x) = \alpha\beta + 2\alpha x^2,
\quad a(x) = g(x) +1 -\alpha (1-x^2), 
$$
we  have to find $\phi_1$ solution of  
\begin{equation}\label{exist0-fi}
\left \{
\begin{array}{l}
a(\phi_0)\phi_1 + \Delta t \Vec{u}_0 \cdot \nabla \phi_1 -\gamma \Delta t \Delta \phi_1 = \phi^0+\phi_0g(\phi_0)
\\
\partial_{\Vec{n}} \phi_1 = 0\quad \mbox{on}\quad \partial \Omega 
\end{array}
\right.
\end{equation}
where $\Vec{u}_0 = \Vec{u}^0$ and $\phi_0 = \phi^0$. Existence and uniqueness of a solution in $\HVolund{3}$ results from Lemma \ref{Lemma_Elliptique}.
However the coercivity needed induces a condition on $\Delta t$, 
namely
\begin{equation*}
a(x) > 0\quad \forall x\in[-1,1] \Leftrightarrow  a(x) \ge a(0)> 0 \Leftrightarrow \alpha(\beta-1) +1 > 0.
\end{equation*}
From the definition of $\alpha$ we get
$$\beta > 1 - \frac{\eta^2}{\gamma\Delta t}.$$
If $\beta\ge 1$ there is no condition on $\Delta t$ on the other hand if 
$\Delta t < \eta^2/\gamma$ then $\beta$ can be an arbitrary positive constant. Therefore existence and uniqueness is obtained for all positive values of $\Delta t$.

Assume $\phi_1$ achieves its maximum and minimum in $\Omega$ at $\Vec{x}^*, \Vec{x}_*$ respectively
$$
\nabla \phi_1(\Vec{x}_*) =\nabla \phi_1(\Vec{x}^*) =\vzero \qquad \mbox{and} \qquad \Delta \phi_1(\Vec{x}_*) \ge 0,\ \Delta \phi_1(\Vec{x}^*) \le 0. 
$$
From \eqref{exist0-fi} at $\Vec{x}_*$,$\Vec{x}^*$, using $|\phi_0(\Vec{x})| \le 1$ and $\beta \ge 0$, we have 
\begin{equation*}
\begin{aligned}
\frac{ \phi^0(\Vec{x_*}) + \phi_0(\Vec{x_*})g(\phi_0(\Vec{x_*})) }{a(\phi_0(\Vec{x_*}))}
\le \phi_1(\Vec{x}_*) 
&\le \phi_1(\Vec{x}) 
\\ &
\le \phi_1(\Vec{x}^*) 
\le \frac{ \phi^0(\Vec{x^*})+\phi_0(\Vec{x^*})g(\phi_0(\Vec{x^*})) }{a(\phi_0(\Vec{x^*}))}
\end{aligned}
\end{equation*}
The role of the parameter $\beta$ is clear, 
it will induce a bound on $\phi_1$ without the need 
for any condition on $\Delta t$. Denoting $\theta = \phi_0(\Vec{x}) = \phi^0(\Vec{x}) $, we want establish the values of $\beta$ giving simultaneously
\begin{equation}\label{lemachin}
-1
\le \frac{-1 + \theta g(\theta) }{a(\theta)}\quad 
\text{and}\quad
\frac{1 + \theta g(\theta)}{a(\theta)}
\le 1\quad\forall \theta\in [-1,1]. 
\end{equation}
Note that the coerciveness of $a(x)$ gives 
a strictly positive denominator.
We introduce 
$$p_-(\theta) = (1 + \theta g(\theta)) - a(\theta)=  \alpha (\theta - 1) (\beta + 2 \theta^2 - \theta - 1)$$
$$p_+(\theta) = (-1 + \theta g(\theta)) + a(\theta)= \alpha (\theta + 1) (\beta + 2 \theta^2 + \theta - 1)$$
These polynomials have complex roots if $\beta \ge 9/8$. 
Under this condition there is only one real root for each of those polynomial and we get
$$p_-(\theta) \le 0 \qquad\forall \theta \le 1 
\qquad\qquad
p_+(\theta) \ge 0 \qquad\forall \theta \ge -1$$
which gives us \eqref{lemachin}. For $\beta < 9/8$ there is no general conclusion regarding these inequalities therefore $\beta \ge 9/8$ is only a sufficient condition for the uniform bound for $\phi_{k+1}$.
%
%



%
%
%

Concerning the Navier-Stokes system \eqref{AlgNSDivu}, notice that at this stage of the loop $\phi_{k+1}\in H^3(\Omega)$, $\rho_{k+1}$ and $\mu_{k+1}$ are known. Therefore the existence, uniqueness (based on the large viscosity assumption) and regularity of the solution $(\Vec{u}_1, p_1)$  
comes form classical results (see  \cite{GirRav1986,QuaVal2008,Tem1979} for instance).  
Proceeding by induction on $k$  for fixed $n=0$, 
we get the existence of $\left(\Vec{u}_{k+1},p_{k+1},\phi_{k+1} \right)$ 
solution of \eqref{AlgAC}--\eqref{AlgNSDivu} with $\phi_{k+1}$ satisfying  
$\|\phi_{k+1}\|_{L^\infty(\Omega)}  \le 1$.

Finally, having proved the assumptions for $n=0$, using again an induction argument over the time iteration $n$ with the same technique as presented here, we get the results for $\left(\Vec{u}_{k+1},p_{k+1},\phi_{k+1} \right)$ solution of \eqref{AlgAC}--\eqref{AlgNSDivu} at time $(n+1)\Delta t$.  
\end{proof}
\begin{remark}
As noted in the proof, the condition on $\beta$ is not a necessity, and we could have a uniform $L^{\infty}$ bound for smaller values of $\beta$. This kind of  "maximum principle" for $\phi_{k+1}$ and $\phi^n$ is fundamental as it defines a phase field on $\Omega$.
\end{remark}
%
%
\subsubsection{Regularity estimates}
In what follows we rely on uniform bounds on $\phi_k$ and $\Vec{u}_k$ to get the convergence of the sequence $( \Vec{u}_k,\phi_k)_{k\ge 0}$ in $\HVoltroisd{2}\times\HVolund{3}$. Applying Lemma \ref{Lemma_Elliptique} to problem \eqref{AlgAC} is not totally satisfactory since we get a $H^2$ bound on $\phi_{k+1}$ depending (implicitly) on $\Vec{u}_k$. However, with some additional work, this lemma can be used to get uniform bound on $\phi_{k+1}$ and $\Vec{u}_{k+1}$. 
\begin{lemma}[Bound on $(\phi_k)_{k\ge 0}$]\label{LM-Est-Gfi-PF}
At each time step $t^{n+1} = (n+1)\Delta t$, for all nonnegative integer $k$ and for all real number $\beta \ge \borneBeta$,
the iterative solution $\phi_{k+1}$ of \eqref{AlgAC} satisfies the following uniform bound 
\begin{equation}\label{Est-Gfi-PF} 
\| \phi_{k+1}\| \le |\Omega|^{1/2} \qquad \mbox{and} \qquad 
\| \nabla \phi_{k+1}\| \le \Cgfi.
\end{equation}
where $\Cgfi$ is a constant depending on $\beta, \eta, \gamma, |\Omega|$ and $\Delta t$.
\end{lemma}
\begin{proof}
The first bound is a direct consequence of $\|\phi_{k+1}\|_\infty \le 1$. 
The second one can be obtained by multiplying \eqref{AlgAC} by $\phi_{k+1}$ and integrating over $\Omega$:
$$
\| \nabla \phi_{k+1}\| \le \frac{|\Omega|^{1/2}}{ \sqrt{\gamma \Delta t}}  
\left(  1+ \frac{\gamma \Delta t}{\eta^2}(\beta+2) \right)^{\frac{1}{2}}
=  \Cgfi. 
$$
\end{proof}
\begin{lemma}[Bound on $(\Vec{u}_k)_{k\ge 0}$]\label{TH-Est-GU-PF}
With the hypothesis of Lemma \ref{LM-Est-Gfi-PF}
the iterative solution $\Vec{u}_{k+1} $of \eqref{AlgNSDivu} satisfies the following uniform bound
\begin{alignat}{2}
\| \Vec{u}_{k+1}\| & \le   \Cu =  \frac{\Delta t |\Omega|^{\frac{1}{2}}}{\rho_I} +
\frac{\rho_S}{\rho_I} \|\Vec{u}^n\| +
\frac{\sigma    \Cgfi}{\rho_I\gamma}. \label{Est-U-PF1}  
\\ 
\| \Vec{D}(\Vec{u}_{k+1})\| & \le \Cgu =   
\sqrt{\frac{\rho_I }{\mu_I}}\frac{\Cu}{\sqrt{\Delta t}}.
\label{Est-U-PF} 
\\ 
\| \Vec{u}_{k+1} \cdot \nabla \phi_{k+1} \|
& \le \Cugfi = 
\sqrt{\frac{\gamma\rho_I}{\sigma}}  \frac{\Cu}{\sqrt{\Delta t}}.
\label{Est-UGFi-PF}
\end{alignat}
\end{lemma}
\begin{proof}
We multiply \eqref{AlgNSDivu} by $\Vec{u}_{k+1}$ and we integrate over $\Omega$, we find 
\begin{alignat*}{2}
\frac{1}{\Delta t}  \int_\Omega \rho_{k+1} &|\Vec{u}_{k+1}|^2 +  \int_\Omega \mu_{k+1}   | \Vec{D}(\Vec{u}_{k+1}) |^2 +
\frac{\sigma}{\gamma} \| \Vec{u}_{k+1} \cdot \nabla \phi_{k+1} \|^2 \\ 
& \quad+ 
\int_\Omega 
\rho_{k+1} \left( \Vec{u}_{k}\cdot \nabla \right) \Vec{u}_{k+1} \cdot \Vec{u}_{k+1}  +
\int_\Omega  \frac{1}{2} \nabla \cdot \left(\rho_{k+1}  \Vec{u}_{k}\right) \Vec{u}_{k+1}\cdot  \Vec{u}_{k+1}
\\ & \quad+ 
\frac{\sigma}{\gamma \Delta t}  
\int_\Omega \left( \Vec{u}_{k+1} \cdot \nabla \right)  \phi_{k+1}  \left( \phi_{k+1} - \phi^n\right)
\\ 
&  =
\int_\Omega \Vec{G}(\phi_{k+1}) \cdot  \Vec{u}_{k+1} +
\frac{1}{\Delta t}\int_\Omega \sqrt{\rho_{k+1}}\sqrt{\rho_{k}} \Vec{u}^{n} \cdot  \Vec{u}_{k+1} . 
\end{alignat*}
By \eqref{rhoUVV},  the second line vanishes and by \eqref{UgradPhiPhi},  the third one equal to 
\begin{align*}
-\frac{\sigma}{\gamma \Delta t } \int_\Omega 
\left( \Vec{u}_{k+1} \cdot \nabla \right)  \phi_{k+1}  \phi^n. 
\end{align*}
Using  \eqref{BorneRhoMu0} and Assumption \ref{HypoG}, yields 
\begin{alignat*}{2}
\frac{\rho_I}{\Delta t} \| \Vec{u}_{k+1}\|^2 + 
\mu_I \| &\Vec{D} (\Vec{u}_{k+1})\|^2 + 
\frac{\sigma}{\gamma} \| \Vec{u}_{k+1} \cdot \nabla \phi_{k+1} \|^2 \\ 
&  \le  
\left( 
\| \Vec{G}(\phi_{k+1})\|  + \frac{\rho_S}{\Delta t} \| \Vec{u}^{n}\|
+ 
\frac{\sigma}{\gamma \Delta t }  \| \phi^n\|_{L^\infty}  \| \nabla \phi_{k+1}\|
\right)  \| \Vec{u}_{k+1}\| \\ 
& \le 
\left( 
|\Omega|^{1/2}  +  \frac{\rho_S}{\Delta t} \| \Vec{u}^{n}\|
+ 
\frac{\sigma}{\gamma \Delta t }   \|\nabla \phi_{k+1}\|
\right)  \| \Vec{u}_{k+1}\|.
\end{alignat*}
Based on Lemma \ref{LM-Est-Gfi-PF} and Korn inequality, we obtain the desired estimates.
\end{proof}

These last lemmas, although giving uniform bounds on $\phi_{k}$, $\Vec{u}_{k}$ and the tensor $\Vec{D}(\Vec{u}_{k})$ do not provide sufficient regularity for the solution of \eqref{AlgAC}--\eqref{AlgNSDivu} to get the convergence results sought for. Based on Lemmas \ref{Lemma_Elliptique}, \ref{LM-Est-Gfi-PF} and estimate \eqref{Est-UGFi-PF} of Lemma~\ref{TH-Est-GU-PF}, we can get uniform bounds in spaces more suited for the rest of the analysis.
\begin{theorem}\label{TheoBorneUPhi}
With the hypothesis of Lemma \ref{LM-Est-Gfi-PF}
the iterative solutions of schemes \eqref{AlgAC}--\eqref{AlgNSDivu} satisfy the following regularity results
\begin{alignat}2
& \|  \phi_{k+1}\|_{2}  \le \Cfid, 
\label{phiH2} \\ 
& \|  p_{k+1}\| \le \Cpc = \Cp,
\label{pL2} \\ &
\|\Vec{u}_{k+1}\|_{2}  +  \left\| \frac{p_{k+1}}{\mu_{k+1}}\right\|_{1}  \le \Cudc = \Cud.  
\label{uH2} 
\end{alignat}
Where $\Cfid, \Komix^p, \Komix^u, \Komou^u$ and $\Komou^f$  are non-negative constants following the convention in section~\ref{sec_Constante}.
\end{theorem}
\begin{proof}
To bound $\phi_{k+1}$ in $H^2(\Omega)$, we rewrite  \eqref{AlgAC}   as follow
\begin{equation}\label{Laplace_Psi}
\left\{
\begin{array}{rccl} 
-\gamma \Delta t \,  \Delta \phi_{k+1}  + c(\Vec{x}) \phi_{k+1}  &=&   f^{k,n}_\phi  &\qquad \mbox{in } \Omega 
\vspace*{4pt}\\ 
\partial_\Vec{n} \phi_{k+1}   &=&  0 &\qquad \mbox{on }  \partial \Omega
\end{array}
\right.
\end{equation}
with
\begin{alignat*}2
1 \le c(\Vec{x})  &=  1+ \frac{\gamma \Delta t}{\eta^2}\left(\beta -1\right) + \frac{3\gamma \Delta t}{\eta^2}  \phi_{k}^{2} (\Vec{x})\in  L^{\infty}(\Omega)  \\
f^{k,n}_\phi &= \phi^n +  \frac{\gamma \Delta t}{\eta^2}\left(\beta+2\phi_k^2\right)\phi_{k}  - \Delta t  u_k \cdot \nabla\phi_{k+1} . 
\end{alignat*}
\noindent Using Lemma \ref{Lemma_Elliptique} with $r=0$ and \eqref{Est-Gfi-PF} we get
\begin{equation}
\begin{aligned} \label{BormeNormeH2Psik1}
\| \phi_{k+1} \|_{2} 
&\le  \Comou \left(  \left( 1 + \frac{\gamma \Delta t}{\eta^2}(\beta+2) \right) |\Omega |^{\frac{1}{2}} +  \Delta t  \| \Vec{u}_k \cdot \nabla\phi_{k+1}   \| +  \Cgfi  \right). 
\end{aligned}
\end{equation}
In order to bound $ \Vec{u}_k \cdot \nabla \phi_{k+1}$ in $L^2(\Omega)$, using a  \SPCSHY
\begin{equation}\begin{aligned} \label{S_kn_t1}
\Delta t\|\Vec{u}_k\cdot \nabla \phi_{k+1}\| &\le  {\Com} \Delta t\| \Vec{u}_k\|_{\Vec{L^6}}  \| \nabla \phi_{k+1}\|_{\Vec{L^3}} \\ 
& \le  {\Com} \Delta t\|\nabla \Vec{u}_k\|  \| \nabla \phi_{k+1}\|^{\frac{6-d}{6}}  \| \nabla \phi_{k+1}\|_{1}^{\frac{d}{6}}\\
& \le {\Com} \alpha^{\frac{d}{6-d}}\left( 
\Delta t  \| \nabla \Vec{u}_{k}\|\right)^{\frac{6}{6-d}} 
\|\nabla \phi_{k+1}\| + \frac{d}{6\alpha} \| \phi_{k+1}\|_{2} \vspace*{4pt}\\ 
& \le {\Com}\alpha^{\frac{d}{6-d}}\left( \Delta t \Cgu\right)^{\frac{6}{6-d}} \Cgfi + \frac{d}{6\alpha} \| \phi_{k+1}\|_{2}\quad\forall \alpha > 0.
\end{aligned}\end{equation}
Then \eqref{phiH2} comes from \eqref{BormeNormeH2Psik1} and \eqref{S_kn_t1}. 
We are now in position to bound $p_{k}$. To do so, 
we first rewrite \eqref{AlgNSDivu} as  a Stokes system with variable viscosity and no-slip condition on $\partial \Omega$ 
\begin{equation}\label{Stokes1}
\left\{
\begin{array}{rccl}
-\nabla \cdot \left(\mu_{k+1} \Vec{D}(\Vec{u}_{k+1}) \right) + \nabla p_{k+1} &=&   \Vec{f}^{k,n}_\Vec{u}(\Vec{u}_{k+1})  &\qquad \mbox{in } \Omega 
\vspace*{4pt}\\ 
\nabla \cdot \Vec{u}_{k+1} &=&  0 &\qquad \mbox{in }  \Omega,
\end{array}
\right.
\end{equation}
where 
\begin{equation}\label{Fk}
\begin{aligned}
\Vec{f}^{k,n}_\Vec{u}(\Vec{v}) &= - \rho_{k+1}\left( \Vec{u}_{k}\cdot \nabla \right) \Vec{v} 
- \frac{1}{2} \nabla \cdot \left(\rho_{k+1} \Vec{u}_{k}\right) \Vec{v} 
-\frac{\sigma}{\gamma} \left( \Vec{v} \cdot \nabla \phi_{k+1} \right) \nabla \phi_{k+1} 
\\ 
&\quad  
- \frac{\rho_{k+1}} {\Delta t} \Vec{v}  - \frac{\sigma}{\gamma \Delta t}\left( \phi_{k+1} - \phi^n\right) \nabla \phi_{k+1}  
+\Vec{G}(\phi_{k+1}) 
\\ 
&\quad  + \sqrt{\rho_{k+1}} \frac{\sqrt{\rho^{n}} \Vec{u}^{n}}{\Delta t} .   
\end{aligned}\end{equation}
Then, we have  
\begin{alignat*}{2}
\|\nabla p_{k+1}\|_{H^{-1}}  & \le  \| \Vec{f}^{k,n}_\Vec{u}(\Vec{u}_{k+1}) \|_{\Vec{H}^{-1}}   
+  \| \nabla \cdot (\mu_{k+1}  \Vec{D}(\Vec{u}_{k+1}) ) \|_{\Vec{H}^{-1}}  \notag \\ 
& \le  \| \Vec{f}^{k,n}_\Vec{u}(\Vec{u}_{k+1}) \|_{\Vec{H}^{-1}}  +  
\mu_S \| \Vec{D}(\Vec{u}_{k+1})  \|.
\end{alignat*}
Since $p_{k+1} $ belongs to $L^2_0(\Omega)$,  Poincar\'e inequality  (see  for instance \cite{BoyFab2013} ) gives
\begin{align}\label{EstPL2}
\| p_{k+1}\|  & \le  \Com \left(
 \| \Vec{f}^{k,n}_\Vec{u}(\Vec{u}_{k+1}) \|_{\Vec{H}^{-1}}  
+ \mu_S \Cgu \right). 
\end{align}
Based on  Lemma \ref{TH-Est-GU-PF}, estimation \eqref{phiH2}  and Cauchy-Schwartz inequality, the first line in \eqref{Fk} can be bounded as follow
\begin{equation*}
\begin{aligned}
\biggl\| \rho_{k+1}( \Vec{u}_{k}\cdot &\nabla) \Vec{u}_{k+1}  
 + \frac{1}{2} \nabla \cdot \left(\rho_{k+1} \Vec{u}_{k}\right) \Vec{u}_{k+1} 
+ \frac{\sigma}{\gamma} \left( \Vec{u}_{k+1} \cdot \nabla \phi_{k+1} \right) \nabla \phi_{k+1}
 \biggr\|_{\Vec{H}^{-1}} \qquad
\\
& \le  \sup_{\substack{\Vec{v} \in \Vec{H}^1_0(\Omega)\\ \|\Vec{v}\|_{1} = 1}} \biggl( 
\rho_S  \int_{\Omega }\left| ( \Vec{u}_{k}\cdot \nabla) \Vec{u}_{k+1} \cdot \Vec{v} \right| 
+ \frac{1}{2}  \int_{\Omega } \left| 
\nabla \cdot \left(\rho_{k+1} \Vec{u}_{k}\right) \Vec{u}_{k+1}  \cdot \Vec{v}  \right| 
\\ 
& \hspace*{0.38\textwidth} + \frac{\sigma}{\gamma} \int_{\Omega} \left|  ( \Vec{u}_{k+1} \cdot \nabla \phi_{k+1} ) \nabla \phi_{k+1} \cdot \Vec{v} \right| \biggr)
\\
& \le \Com \left( \rho_S (\Cgu)^2 +   \frac{\sigma}{\gamma}  \Cugfi \, \Cfid
\right) = \Komix_1.
\end{aligned}
\end{equation*}
As for the second and third line in \eqref{Fk}, it can be bounded by 
\begin{equation*}\begin{aligned}
\biggl\| \frac{\rho_{k+1}} {\Delta t} \Vec{v}  + \frac{\sigma}{\gamma \Delta t}( \phi_{k+1} - &\phi^n)\nabla \phi_{k+1}  
-\Vec{G}(\phi_{k+1}) 
- \sqrt{\rho_{k+1}} \frac{\sqrt{\rho^{n}} \Vec{u}^{n}}{\Delta t} \biggr\|_{\Vec{H}^{-1}} 
\\ 
& \le \Com \left( 
\frac{\rho_S}{\Delta t}  \Cgu + \frac{\sigma }{\gamma \Delta t} \Cgfi + |\Omega|^{\frac{1}{2}} 
+ \frac{\rho_S}{\Delta t} \| \Vec{u}^n\|
\right) = \Komix_2.
\end{aligned}
\end{equation*}
Finally 
\begin{equation}\label{PkFka}
\| p_{k+1}\| \le \Com(\Komix_1+\Komix_2+\mu_S \Cgu) = \Cp = \Cpc.\end{equation}
To bound  $ \|\Vec{u}_k\|_{2}$, we  first  rearrange \eqref{Stokes1}  using $\hat p_{k+1} = p_{k+1}/ \mu_{k+1}$
\begin{equation}\label{Stokes-VarViscosity}
\left\{
\begin{aligned}
-&\nabla \cdot \Vec{D}(\Vec{u}_{k+1})  + \nabla \hat p_{k+1} =  \displaystyle
\frac{1}{\mu_{k+1}} \biggl(  \Vec{D}(\Vec{u}_{k+1}) \cdot \nabla {\mu_{k+1}}
\\ & \qquad\qquad\qquad\qquad\qquad\qquad\qquad
+   \Vec{f}^{k,n}_\Vec{u}(\Vec{u}_{k+1})  + \hat p_{k+1} \nabla \mu_{k+1} \biggr) 
\\
& \nabla \cdot \Vec{u}_{k+1} =  0, 
\end{aligned}
\right.
\end{equation}
where $\Vec{f}^{k,n}_\Vec{u} $ is defined by \eqref{Fk}. 
Next,  following the regularity result for Stokes problem (see for instane  \cite{BoyFab2013}), 
we have  
\begin{equation}\label{UkFk0}
\begin{aligned}
\| \Vec{u}_{k+1}&\|_{2}  + \left\| \hat p_{k+1} \right\|_{1}  
\\
 & \le 
\frac{\Com}{\mu_I} \left(
 \| \Vec{D}(\Vec{u}_{k+1}) \cdot \nabla \mu_{k+1}\|
+  \| \hat p_{k+1} \nabla \mu_{k+1}\|   
+  \| \Vec{f}^{k,n}_\Vec{u}(\Vec{u}_{k+1})\|\right).
\end{aligned}
\end{equation}
The first two terms on the right hande side can be bounded (using  a  \SPCSHY)  as follow  
\begin{alignat*} 2
\|\Vec{D} (\Vec{u}_{k+1} )\cdot \nabla \mu_{k+1}\| & \le 
{\Comi{1}} \dmu  \|\phi_{k+1}\|_{2}\|\Vec{D} (\Vec{u}_{k+1})\|_{\Vec{L}^3} \notag 
\\ &
\le {\Comi{1}} \dmu  \|\phi_{k+1}\|_{2}\|\Vec{D} (\Vec{u}_{k+1})\|^{\frac{6-d}{6}} 
\|\Vec{D} (\Vec{u}_{k+1})\|_{1} ^{\frac{d}{6}}  \notag \\ 
& \le  
{\Comi{1}} \dmu \Cfid (\Cgu)^{\frac{6-d}{6}} 
\|\Vec{u}_{k+1}\|_{2} ^{\frac{d}{6}}.
\end{alignat*}
Then
\begin{equation}\begin{aligned}\label{DkNablaFi}
\frac{\Comi{0}}{\mu_I} \|\Vec{D} \Vec{u}_{k+1} \cdot \nabla \mu_{k+1}\| & \le 
\Com  \left(\frac{\dmu }{\mu_I}\right)^\frac{6}{d}   (\Cfid)^\frac{6}{d} (\Cgu)^\frac{6-d}{d} 
+ \frac{d}{6  } \| \Vec{u}_{k+1}\|_{2}.
\end{aligned}\end{equation}
And
\begin{equation*}\begin{aligned}
\left\| \frac{p_{k+1}}{\mu_{k+1}} \nabla \mu_{k+1}\right\|  
&\le \Com \dmu  \| \phi_{k+1}\|_{2}  \left\| \frac{p_{k+1}}{\mu_{k+1}} \right\|_{L^3}  
\\ & 
\le \Com \dmu  \Cfid  \left\| \frac{p_{k+1}}{\mu_{k+1}} \right\|^\frac{6-d}{d}  \left\| \frac{p_{k+1}}{\mu_{k+1}} \right\|_{1}^\frac{d}{6}.
\end{aligned}\end{equation*}
Then 
\begin{equation}\begin{aligned}\label{PhatH1}
\frac{\Comi{0}}{\mu_I} \biggl\| \frac{p_{k+1}}{\mu_{k+1}} &\nabla \mu_{k+1}\biggr\|  
\le \Com \dmu^\frac{6}{d} \mu_I^\frac{d-12}{d}   
(\Cfid)^\frac{6}{d}  (\Cpc)^\frac{6-d}{d}  
+ \frac{d}{6}  \left\| \frac{p_{k+1}}{\mu_{k+1}}\right\|_{1}.
\end{aligned}\end{equation}
Combining \eqref{PhatH1}, \eqref{DkNablaFi} with  \eqref{UkFk0} there is two constants $\Komix^u$ and $\Komou^u$ such that,
\begin{equation}
\begin{aligned}\label{UkFk}
\| \Vec{u}_{k+1}\|_{2}  + \left\| \frac{p_{k+1}}{\mu_{k+1}} \right\|_{1}   
&\le \Cudpasf \\
& \qquad  \qquad 
+ \frac{\Com}{\mu_I} \| \Vec{f}^{k,n}_\Vec{u}(\Vec{u}_{k+1})\|.
\end{aligned}
\end{equation}
Going back to~\eqref{Fk}, each term in $\Vec{f}^{k,n}_\Vec{u}$ can be bounded using   \eqref{Est-U-PF}, \eqref{phiH2} and applying a  \SPCSHY.  For $\alpha>0$
\begin{equation*}\label{Reg-UgU}
\begin{aligned}
\| \rho_{k+1} \left( \Vec{u}_{k}\cdot \nabla \right) \Vec{u}_{k+1}  \| & \le 
\rho_{S}  \|\left( \Vec{u}_{k}\cdot \nabla \right) \Vec{u}_{k+1}  \| 
\le \rho_{S}  \|\Vec{u}_{k}\|_{\Vec{L}^6}  \|\nabla \Vec{u}_{k+1}  \|_{\Vec{L}^3} 
\\ & 
\le {\Comi{1}} \rho_S \|\nabla \Vec{u}_{k}\| \|\nabla \Vec{u}_{k+1}  \|^{\frac{6-d}{6}} \|\nabla \Vec{u}_{k+1}  \|^{\frac{d}{6}}_{\Vec{1}}  
\\ & 
\le {\Comi{1}} 
\alpha^{\frac{d}{6-d}}{\rho_S}^{\frac{6}{6-d}}  \|\nabla \Vec{u}_{k}\|^{\frac{6}{6-d}} \|\nabla \Vec{u}_{k+1}  \| + \frac{d}{6\alpha} \|\Vec{u}_{k+1}  \|_{2} 
\\ &
\le \Comi{1}\alpha^{\frac{d}{6-d}}\rho_S^{\frac{6}{6-d}} \left(\Cgu\right)^{\frac{12-d}{6-d}} + \frac{d}{6\alpha} \|\Vec{u}_{k+1}  \|_{2}.
\end{aligned}\end{equation*}
For the second term (as $\Vec{u}_k$ is divergence free)
\begin{equation*}\label{Reg-DivUU}
\begin{aligned}
\| \nabla \cdot \left(\rho_{k+1}  \Vec{u}_{k} \right) \Vec{u}_{k+1}  \| & 
\le  \| \Vec{u}_{k+1}  \|_{\Vec{L}^6}  \| \nabla \cdot \left(\rho_{k+1}  \Vec{u}_{k} \right)\|_{L^3} 
\\ & 
\le {\Comi{2}}\| \nabla \Vec{u}_{k+1}  \| \| \nabla \rho_{k+1}  \|_{\Vec{L}^6} \| \Vec{u}_{k}  \|_{\Vec{L}^6}
\\ & 
\le  {\Comi{2}} \dro \| \nabla \Vec{u}_{k+1}  \|\| \nabla \Vec{u}_{k}  \|  \| \nabla \phi_{k+1}  \|_{1}
\\
& \le  \Comi{2}\dro\left(\Cgu\right)^2 \Cfid.
\end{aligned}\end{equation*}
For the third term 
\begin{equation*}\label{uGfiGfi}
\begin{aligned}
\| \left( \Vec{u}_{k+1} \cdot \nabla \phi_{k+1} \right)  \nabla \phi_{k+1} \|  & 
\le \|  \Vec{u}_{k+1} \|_{\Vec{L}^6}  \| \nabla \phi_{k+1} \|^2_{\Vec{L}^6}  
\le \Cgu (\Cfid)^2.
\end{aligned}
\end{equation*}
Finally 
\begin{equation*}\begin{aligned}\label{BoutenG}
\biggl\| - \frac{\sigma}{\gamma \Delta t}\left( \phi_{k+1} - \phi^n\right) \nabla \phi_{k+1} & + \Vec{G}(\phi_{k+1})+ \sqrt{\rho_{k+1}} \frac{\sqrt{\rho^{n}} \Vec{u}^{n}}{\Delta t}  
-  \rho_{k+1} \frac{\Vec{u}_{k+1} }{\Delta t}\biggr\|
\\ 
&\le 
\left(2\frac{\sigma\,  \Cgfi}{\gamma  \Delta t}
+  |\Omega|^{\frac{1}{2}} 
+  2\frac{\rho_S  \Cu}{\Delta t}\right). 
 \end{aligned}
 \end{equation*}
%
Then, combining the inequalities of each terms, there is a $\alpha$ and a constant $\Komou^f$ such that
\begin{equation}\label{KFk}
\frac{\Com}{\mu_I}\| \Vec{f}^{k,n}_\Vec{u}(\Vec{u}_{k+1}) \|  
\le 
\Komou^f + \frac{1}{2} \|\Vec{u}_{k+1}  \|_{2}
\end{equation}

\noindent and from \eqref{UkFk}--\eqref{KFk}, 
\begin{equation}
\begin{aligned}\label{UkFka}
\| \Vec{u}_{k+1}\|_{2}  + \biggl\| \frac{p_{k+1}}{\mu_{k+1}}& \biggr\|_{1}   
\le  \Cud 
\end{aligned}
\end{equation}
which completes the proof.
\end{proof}

These uniform bounds suffice to conclude to the existence of a converging subsequence with very few constraints on the physical data. However, we are interested in establishing the convergence of the algorithm composed of the loop \eqref{AlgAC}--\eqref{AlgNSDivu}.
\subsubsection{Convergence analysis and passage to the limit}\label{Sec-CVG}
For the next proofs, we introduce the following quantities 
$$
\alpha_0=\left(1 +  \frac{\gamma \Delta t}{\eta^2}(\beta -1) \right),\qquad 
\alpha_1= \frac{ \gamma \Delta t}{\eta^2}\qquad \mbox{and} \qquad
\alpha_2= \frac{ \gamma \Delta t}{\eta^2}(\beta+12).
$$
It is clear that  $1 \le \alpha_0 $ for $\beta \ge \borneBeta$. \\

\begin{theorem}\label{th-Est-fikm}
Let  $(\Vec{u}_k,\phi_k)_{k\ge 0}$ be the sequence of solutions of \eqref{AlgAC}--\eqref{AlgNSDivu} at time $t^{n+1}$. Assuming $\beta \ge \borneBeta$
then   
for all nonnegative integers $m,k$:
\begin{alignat}{2}
\| \phi_{k+1}-\phi_{m+1} \| & \le   
\mathscr{C}_{\phi0}  \left( \|\phi_{k}  - \phi_{m} \|  + \|\nabla\left(\Vec{u}_{k}  - \Vec{u}_{m} \right)\|\right).
\label{Est-fikm} \\
\|\nabla \left( \Vec{u}_{m+1} - \Vec{u}_{k+1} \right) \| 
& \le 
K_{u0} (\| \phi_{m} -  \phi_{k}\| + \| \nabla \left(\Vec{u}_{m} -  \Vec{u}_{k} \right) \|).
\label{Est-ukm} 
\end{alignat}
Where $\mathscr{C}_{\phi0}$ depends only on  $\gamma, \eta,\beta$ and $\Delta t$. 
\end{theorem}
\begin{proof}
To simplify the presentation we introduce the notation 
$$\dfi = \phi_{k+1}-\phi_{m+1}, \qquad \du = \Vec{u}_{m+1}-\Vec{u}_{k+1},\qquad \delta\Vec{u}^p = \Vec{u}_{m}-\Vec{u}_{k}$$
The proof will be split in two parts, beginning by the analysis of the sequence $\phi_k$.  
\\

\noindent\textit{\textbf{Analysis of the sequence $(\phi_k)_{k \ge 0}$.}}
\noindent Taking \eqref{AlgAC} at two different iterations, $k$ and $m$,  multiplying each equalities by $\delta\phi$ and subtracting them, we get
\begin{equation*}\begin{aligned}
\alpha_0  \|\delta\phi \|^2 &+\gamma \Delta t \| \nabla \delta\phi \|^2 
+\frac{3\gamma \Delta t}{\eta^2}  \int_\Omega\left( \phi_k^2\phi_{k+1}-\phi_m^2 \phi_{m+1} \right) \dfi 
\\ & 
= 
\Delta t  \int_\Omega \left( \Vec{u}_k \cdot  \nabla\phi_{k+1} -  
\Vec{u}_m \cdot \nabla \phi_{m+1} \right)\dfi 
\\ & \qquad 
+ \frac{\gamma \Delta t}{\eta^2}\beta  
\int_\Omega \left( \phi_{k} - \phi_{m}  \right) \dfi
+ \frac{2\gamma \Delta t}{\eta^2}
\int_\Omega \left( \phi_{k}^3 - \phi_{m}^3  \right) \dfi.
\end{aligned}\end{equation*}
From the bound on $(\phi_k)_k$,  we have
$$
|\left(\phi_k^2-\phi_m^2\right) \phi_{m+1}\dfi| 
\le 2 |\left(\phi_{k}  - \phi_{m}\right)\dfi|,
\quad |\left(\phi_k^3-\phi_m^3\right) \dfi| \le 
3 |\left(\phi_{k}  - \phi_{m}\right)\dfi|,
$$
and using \eqref{UgradPhiPhi1} for the term in $\Vec{u}$ we have
\begin{equation}\begin{aligned}\label{Inegalite_Phi_L2}
\alpha_0 \|&\dfi \|^2 + \gamma \Delta t \| \nabla \dfi \|^2
\le  \alpha_2\| \phi_m - \phi_k \|  \| \dfi\|
+  \Delta t \|\Vec{u}_k -\Vec{u}_m\|\|\nabla\dfi\|.\vspace*{6pt}
\end{aligned}\end{equation}
From which we get
\begin{equation}\label{ineg_fikm_l2}
\|\phi_{m+1}-\phi_{k+1} \| \le  
\frac{\alpha_2}{\alpha_0} \| \phi_m - \phi_k \|+ 
\sqrt{\frac{\Delta t}{\alpha_0 \gamma}}  \|\Vec{u}_k -\Vec{u}_m\|
\end{equation}
and \eqref{Est-fikm} follows by choosing
$
\mathscr{C}_{\phi0} = \displaystyle\max\left(\frac{\alpha_2}{\alpha_0}, \sqrt{\frac{\Delta t}{\alpha_0 \gamma}}\right).
$

In same manner we get a bound on the  gradient of the difference
and a $H^1$ bound 
\begin{equation}\label{Inegalite_Phi_H1}
\| \phi_{k+1}  - \phi_{m+1} \|_{1} \le \mathscr{C}_{\phi1} \left(  \| \phi_m - \phi_k \| + \|\Vec{u}_k -\Vec{u}_m\| \right),
\end{equation}
with 
$$\mathscr{C}_{\phi1} = \max\left(\mathscr{C}_{\phi0},\frac{\alpha_2}{ \sqrt{2\alpha_0 \gamma \Delta t}}, \frac{1}{\gamma}\right).$$\\

\noindent\textit{\textbf{Analysis of the sequence $(\Vec{u}_k)_{k \ge 0}$.}} We multiply \eqref{AlgNSDivu} at $(k+1)$ and $(m+1)$ steps by $\du$ and we integrate over $\Omega$,  yielding a sum of seven terms
\begin{equation*}
\begin{aligned}
0 = \frac{1}{\Delta t}&  
\int_\Omega 
\left( 
\rho_{m+1} \Vec{u}_{m+1}  -  
\rho_{k+1} \Vec{u}_{k+1}
\right)  \cdot
\du
\\ & 
+\int_\Omega \left( 
\mu_{m+1}   \Vec{D}(\Vec{u}_{m+1})  -  
\mu_{k+1}   \Vec{D}(\Vec{u}_{k+1}) \right) :
\Vec{D} (\du)  
\\ &
+ \frac{\sigma}{\gamma} \int_\Omega \left[ \left( \Vec{u}_{m+1} \cdot \nabla \phi_{m+1} \right)  \nabla \phi_{m+1} -
\left( \Vec{u}_{k+1} \cdot \nabla \phi_{k+1} \right)  \nabla \phi_{k+1}
\right] \cdot   \du
\\ &
+\frac{\sigma}{\gamma \Delta t}\int_\Omega \left[  
\left( \phi_{m+1} - \phi^n\right) \nabla \phi_{m+1}  - 
\left( \phi_{k+1} - \phi^n\right) \nabla \phi_{k+1} 
\right]  \cdot  \du
\\ &
+ \int_\Omega \left( 
\rho_{m+1} \left( \Vec{u}_{m}\cdot \nabla \right) \Vec{u}_{m+1}  - 
\rho_{k+1} \left( \Vec{u}_{k}\cdot \nabla \right) \Vec{u}_{k+1} 
\right)  \cdot \du 
\\ &
+ \frac{1}{2}  \int_\Omega  
\left( \nabla \cdot \left(\rho_{m+1}  \Vec{u}_{m}\right) \Vec{u}_{m+1}  -
\nabla \cdot \left(\rho_{k+1}  \Vec{u}_{k}\right) \Vec{u}_{k+1}
\right) \cdot \du 
\\ &
- \int_\Omega \left(\left(  \Vec{G}(\phi_{m+1})- \Vec{G}(\phi_{k+1}) \right)  
-
\frac{ \sqrt{\rho^{n}}}{\Delta t}  \left( 
\sqrt{\rho_{m+1}} -   \sqrt{\rho_{k+1}} \right) \Vec{u}^{n}\right)\cdot\du 
. \end{aligned}\end{equation*}
To establish \eqref{Est-ukm} we will take advantage of the fact that the first three terms of this sum contains squared norms. Labeling the terms $A_0$ through $A_6$ we have
$$A_0 +A_1 + A_2 = -\sum\limits_{i=3}^6 A_i.$$
Adding  and subtracting $\rho_{m+1} \Vec{u}_{k+1} $ in $A_0$ and using \eqref{AlgRhoMu},
\begin{equation*}\begin{aligned}
A_0 &= \frac{1}{\Delta t}\int_\Omega \rho_{m+1} | \du|^2 
+ \frac{\dro}{2\Delta t} \int_\Omega \dfi\,
\Vec{u}_{m+1}\cdot \du\\
& = \frac{1}{\Delta t}\int_\Omega \rho_{m+1} | \du|^2 + A_{01}
\end{aligned}\end{equation*}
Using \eqref{Cauchy_laplace} we have
\begin{equation*}\begin{aligned}
A_1 &= \int_\Omega \mu_{m+1}|\Vec{D}(\du)|^2 
+\frac{\dmu}{2} \int_\Omega \dfi\,
\Vec{D}(\Vec{u}_{m+1}) :\Vec{D} (\du)\\
& = \int_\Omega \mu_{m+1}|\Vec{D}(\du)|^2  + A_{11}.
\end{aligned}\end{equation*}
We can rewrite $A_2$ as follow:
\begin{equation}\begin{aligned} 
A_2 &= 
\frac{\sigma}{\gamma} \int_\Omega \left( \Vec{u}_{m+1} \cdot \nabla \phi_{m+1} \right) 
\nabla \dfi \cdot   \du 
+\frac{\sigma}{\gamma} \int_\Omega \left( \Vec{u}_{m+1} \cdot \nabla  \dfi \right) \nabla  \phi_{k+1} \cdot   \du 
\\ &  \qquad \qquad \qquad 
+ \frac{\sigma}{\gamma} \int_\Omega \left| \du \cdot \nabla  \phi_{k+1} \right|^2
\\  & 
=  A_{21} + \frac{\sigma}{\gamma}  \| \du \cdot \nabla  \phi_{k+1} \|^2.
\end{aligned}\end{equation}
From these equality we get
\begin{equation}\label{SommeAi}
\mu_I\|\nabla\du\|^2 \le  |\sum_{i=3}^6  A_i| + |A_{01}| + |A_{11}| + |A_{21}|.
\end{equation}
We will now establish an appropriate bound for the right hand side of this last inequality. Using \textit{\SPCSHY s},
\begin{equation}\begin{aligned}\label{A02}
|A_{01}| &\le \frac{\dro}{\Delta t}  {\Com} \|\dfi\|_{1} 
\|\nabla \du\| \| \nabla \Vec{u}_{m+1} \|
\\ & 
\le  \frac{\dro}{\Delta t} \Com  \Cgu\|\dfi\|_{1} 
\|\nabla \du\|. 
\\ &
\\ 
|A_{11}| & \le  \frac{\dmu}{2}  \|\dfi\|_{L^4} 
\| \Vec{D}(\Vec{u}_{m+1}) \|_{\Vec{L}^4} \| \Vec{D} (\du)\|
\\ & 
\le \dmu \Com\| \dfi\|_{1} 
\| \nabla \Vec{u}_{m+1} \|_{\Vec{L}^4} \| \nabla\du\| 
\\ & 
\le \dmu {\Com}  \Cudc \| \dfi\|_{1} \| \nabla\du\|.
\\ &
\\
|A_{21}|  &\le  \displaystyle \frac{\sigma}{\gamma} 
\| \Vec{u}_{m+1}\|_{\Vec{L}^6} \| \nabla   \phi_{m+1}  \| _{\Vec{L}^6}  
\| \nabla \dfi\| \| \du\| _{\Vec{L}^6}
\\ &
\le  \frac{\sigma}{\gamma}{\Com} \Cfid \Cgu  
\|\dfi\|_{1} \| \nabla\du\|. 
\end{aligned}\end{equation}
For $A_3$, using \eqref{UgradPhiPhi} 
\begin{equation}\label{A3}
\begin{aligned}
|A_3| &\le
\frac{\sigma}{\gamma \Delta t}\int_\Omega |\nabla \dfi \cdot  \du\, \phi^n| 
\le {\Com} \frac{\sigma}{\gamma \Delta t}  \|\dfi\|_{1}  
\|\nabla \du  \|.
\end{aligned}\end{equation}
For $A_4+A_5$, we use \eqref{Cauchy_convect}, moreover since all velocities are divergence free we get parts of both terms canceling each others, leaving
\begin{equation*}\begin{aligned} 
A_4 + A_5 &=  
\int_\Omega \left(\rho_{m+1} \left( \dup\cdot \nabla \right) \Vec{u}_{m+1} \right)  \cdot\du 
+ \frac{\dro}{2}\int_\Omega \left( \dfi
\left(  \Vec{u}_{k} \cdot \nabla \right) \ \Vec{u}_{k+1} \right)  \cdot\du
\\ &\qquad 
+\frac{1}{2}  \int_\Omega 
\nabla \rho_{m+1} \cdot \dup   
\Vec{u}_{m+1}  \cdot \du 
+ \frac{1}{2}
\frac{\dro}{2}  \int_\Omega 
(\nabla \dfi \cdot \Vec{u}_{k}) 
(\Vec{u}_{k+1} \cdot \du).
\end{aligned}\end{equation*}
For those four terms we get (through \textit{\SPCSHY s})
\begin{equation}\begin{aligned}\label{I1}
\int_\Omega  
\rho_{m+1}( \left( \Vec{u}_{m}  -  \Vec{u}_{k} \right)\cdot \nabla)   \Vec{u}_{m+1}
\cdot\du 
&\le \rho_S {\Com} \| \nabla \dup \| 
\| \nabla \du \| \| \nabla \Vec{u}_{m+1}\| \\ 
& \le \rho_S {\Com} \Cgu \| \nabla \dup \| \| \nabla \du \|. 
\end{aligned}\end{equation}
\begin{equation}\begin{aligned}\label{I2}
\frac{\dro}{2}\int_\Omega (\dfi\,(\Vec{u}_{k} \cdot \nabla) \Vec{u}_{k+1})  \cdot \du
& \le \frac{\dro}{2} \|\dfi\|_{L^6} 
 \|  \Vec{u}_{k}\|_{\Vec{L^6}} \| \nabla  \Vec{u}_{k+1}\| \|  \du \|_{\Vec{L^6}}
\\ & 
\le \dro {\Com} \|\dfi\|_{1} 
\|  \nabla \Vec{u}_{k}\| \| \nabla  \Vec{u}_{k+1}\| \|   \nabla \du \|  
\\ & 
\le \dro {\Com}\left(\Cgu\right)^2
\|\dfi \|_{1} \| \nabla \du \|.
\end{aligned}\end{equation}
\begin{equation}\begin{aligned}\label{J1}
\frac{1}{2}  \int_\Omega \nabla \rho_{m+1}\cdot \dup\,\Vec{u}_{m+1}  \cdot \du 
&\le \frac{\dro}{4} \| \nabla \phi_{m+1} \|  
\| \Vec{u}_{m+1} \|_{\Vec{L}^6} 
\| \dup \|_{\Vec{L}^6} 
\| \du \|_{\Vec{L}^6} 
\\  &
\le \dro {\Com}  \Cgfi \Cgu \| \nabla \dup \| \| \nabla \du\|. 
\end{aligned}\end{equation}
\begin{equation}\begin{aligned}\label{J2}
\frac{\dro}{4}  \int_\Omega 
(\nabla\dfi &\cdot \Vec{u}_{k}) 
(\Vec{u}_{k+1} \cdot \du) 
\le  \dro {\Com} \left(\Cgu\right)^2 
\|\dfi \|_{1} \| \nabla \du\|. 
\end{aligned}\end{equation}
From \eqref{I1}--\eqref{J2}, we have 
\begin{equation}\begin{aligned}\label{A4+A5}
|A_4 +A_5|  \le \dro {\Com} &\left(\Cgu\right)^2 \|\dfi \|_{1}\|\nabla \du \|
\\ & \quad 
+  {\Com} \Cgu (\rho_S+ \dro \Cgfi) \| \nabla \dup \| \| \nabla \du\|.
\end{aligned}\end{equation}
From Assumption \ref{HypoG} and $ \displaystyle \left| \sqrt{\rho_{m+1}}-  \sqrt{\rho_{k+1}} \right|  \le  
\frac{\dro}{4 \sqrt{\rho_I}} \left|\dfi\right|, \, \forall (k,m) \in \mathbb{N}^2$,
\begin{equation}\begin{aligned}\label{A7}
|A_6| & \le 
\Com \|\dfi\|_{1} \| \nabla \du\|
+
\frac{\dro}{4 \Delta t } \sqrt{ \frac{\rho_S}{\rho_I}}
\int_\Omega |\dfi|  |\Vec{u}^{n} \cdot  (\du)|
\\ &
\le  \Com \left(1
+ \frac{\dro}{\Delta t} \sqrt{\frac{\rho_S}{\rho_I}}  
\Cgu \right)
\|\dfi\|_{1} \| \nabla \du\|.
\end{aligned}\end{equation}
Combining \eqref{A02}--\eqref{A3},\eqref{A4+A5}--\eqref{A7} in \eqref{SommeAi}
and introducing the notation
\begin{equation}\label{Kufi}
\begin{aligned}
K_{u\phi} &= {\Com}  \Biggl(1 + \frac{\sigma}{\gamma  \Delta t} + \dmu\Cudc
\\ & \quad\quad\quad\quad\qquad
+ \Cgu \left(\frac{\sigma}{\gamma}\Cfid +\frac{\dro}{\Delta t} 
+ \frac{\dro}{\Delta t}\sqrt{\frac{\rho_S}{\rho_I}} 
+ \dro\Cgu \right)\Biggl), 
\\
\mathscr{K}_{uu} &=  {\Com} \Cgu (\rho_S+ \dro \Cgfi),
\end{aligned}\end{equation}
we get
\begin{equation}\label{Est-ukm1}
\|\nabla \left( \Vec{u}_{m+1} - \Vec{u}_{k+1} \right) \| \le 
\frac{K_{u\phi}}{\mu_I} \| \phi_{m+1}-\phi_{k+1}\|_{1}  
+ \frac{\mathscr{K}_{uu}}{\mu_I}  \| \nabla \left(\Vec{u}_{m} -  \Vec{u}_{k} \right) \|.
\end{equation}
Combining \eqref{Est-fikm}, \eqref{Est-ukm1}, and \eqref{Inegalite_Phi_H1}, we get \eqref{Est-ukm} with
$$K_{u0} =  \frac{K_{u\phi}\mathscr{C}_{\phi1}+ \mathscr{K}_{uu}}{\mu_I}.$$
\end{proof}

The rest of this analysis is based on the fact that, using \eqref{Est-fikm} and \eqref{Est-ukm}, 
we can clearly provide conditions on $\mathscr{C}_{\phi0}$ and $K_{u0}$ such that the sequence $\left(\Vec{u_k},\phi_k\right)$  converges strongly in $\HVoltroisd{1}\times \Lund{2}$. 
This relies on our ability to control those constants. 

For $\mathscr{C}_{\phi0}$ as it depends on two arbitrary parameters, $\beta$ and $\Delta t$ its control is relatively simple. As for $K_{u0}$, using $\mu_I$ seems to be the simplest route. However this means to impose a lower bound on the minimal value of the viscosity. Let us underline that adding such condition is not specific to the approach proposed here and can be found in other papers such as \cite{ChaDelYak2010,ChaYak2017,DetJenYak2014}. Furthermore, imposing a lower bound on the viscosity,  as a condition for existence and uniqueness of solution, is relatively frequent even in the stationnary case~\cite[Chap. 3]{BoyFab2013}.

This leads us to study the behaviour of $K_{u0}$ with respect to $\mu_I$. From the convention defining the constants (section \ref{sec_Constante}), 
$\mathscr{C}_{\phi1}$ do not depend on $\mu_I$. As for $\mathscr{K}_{uu}$ it is a linear expression with respect to $1/\sqrt{\mu_I}$. Which leaves $K_{u\phi}$  to examine. From \eqref{PkFka}, \eqref{UkFka}, and \eqref{Kufi}, we get, again based on the convention for the constants,
\begin{equation*}
\begin{aligned}
\frac{K_{u\phi}}{\mu_I} 
= \Komou_1 + 
\left( \frac{\mu_S }{\mu_I}-1\right)
\left(
\left( \frac{\mu_S }{\mu_I}-1\right)^\frac{6}{d}   
\left(\left(\Komou_2 +\frac{\mu_S}{\mu_I}\Komou_3\right)^\frac{6-d}{d}
 + \Komou_4 \right) + \Komou_5
\right)
\end{aligned}
\end{equation*}
From this we can conclude that $K_{u0}$ goes to zero when $\mu_I$ goes to infinity if 
\begin{equation}\label{limiteCfiCu}
\lim\limits_{\mu_I\to\infty} \mu_S/\mu_I = cst
\end{equation}
\begin{remark}\label{borneMu}
The condition \eqref{limiteCfiCu} is quite general, a simpler, less generic, condition can be considered. To have
$$K_{u0} < 1$$
we need to assume that $\mu_I$ is sufficiently large but also that 
$\mu_S/\mu_I$
is small enough. 
Therefore it is sufficient to
assume that there is a $\Com_\mu >0$ such that
\begin{equation}\label{bornedeltaMu}
\mu_S - \mu_I < \Com_\mu{\mu_I}\qquad \forall \mu_I.
\end{equation}
to make sure that we can control $K_{u0}$ through $\mu_I$. 
\end{remark}
With \eqref{limiteCfiCu} or \eqref{bornedeltaMu}, we get, for $\mu_I$ sufficiently large, the convergence of the \textit{whole}  sequence $\left(\Vec{u_k},\phi_k\right)$  
in $\HVoltroisd{1}\times \Lund{2}$ (and in $ \HVoltroisd{1}  \times \HVolund{1}$). 
This will be the subject of Corollary \ref{convergenceH1}.

\begin{corollary}\label{convergenceH1}
Under the assumptions of Theorem~\ref{Existence}, with $\Vec{u}^0$ divergence free. 
Assuming $\mu_I$ is sufficiently large and
the following conditions are satisfied
\begin{equation}\label{CondDeltat1}
\Delta t < \frac{\eta^2}{13\gamma}
,\quad \beta 
>\max\{\borneBeta, \frac{\eta^2}{(1-K_{u0})\gamma^2} -12\},
\end{equation}
then, at each time step, the sequence $(\Vec{u}_k, \phi_k)_{k\ge 0}$ of solutions of \eqref{AlgAC}--\eqref{AlgNSDivu} converge strongly in $ \HVoltroisd{1}  \times \HVolund{1}$ and for all integers $k$,$m$ 
\begin{equation}\label{borneGeometrique}
\|\phi_{k+1}-\phi_{m+1}\| + \|\nabla(\Vec{u}_{k+1}-\Vec{u}_{m+1}) \| \le K(\|\phi_{k}-\phi_{m}\| + \|\nabla(\Vec{u}_{k}-\Vec{u}_{m}) \|)
\end{equation}
where \begin{equation}\label{kmagique}
K= \left(\mathscr{C}_{\phi 0}+K_{u0}\right) < 1.
\end{equation}
\end{corollary}
\begin{proof}
Since $\beta \ge 9/8$ Theorem \ref{th-Est-fikm} is valid, adding \eqref{Est-fikm}, \eqref{Est-ukm}, we have \eqref{borneGeometrique}.
It is relatively easy to show that $K<1$ and $(\Vec{u}_k, \phi_k)$ converges if the conditions on $\beta$ and $\Delta t$ are satisfied. 
\end{proof}

\begin{remark} 
Since we assume that $\eta^2/\gamma \ll 1$ (i.e. $\eta$ is nearly (or goes) to zero)
in most case the condition $\beta \ge \borneBeta$ is sufficient.
For a fixed time step, as $\mathscr{C}_{\phi 0} > 0$, $\beta$ can be seen as an "acceleration coefficient", as it permits to lower the convergence rate of the sequence.
%
%
%
%
Alternatively, if $K_{u0}$ approches $1$ then $\beta$ should be large to enforce the bound on the phase field and the convergence of the sequence. 
\end{remark}

\begin{theorem}\label{convergenceH2}
Under the hypothesis of Corollary \ref{convergenceH1}, at each time step $t^n = n\Delta t$, the sequence $(\Vec{u}_k, p_k, \phi_k)_{k\ge 0}$ of solutions of \eqref{AlgAC}--\eqref{AlgNSDivu} converge strongly in 
$ \HzVoltroisd \cap \HVoltroisd{2} \times L^2_0(\Omega)\cap H^1(\Omega) \times H^3(\Omega)$ to a solution $(\Vec{u}^n, p^n,\phi^n)$ of \eqref{EqSysSD}, \eqref{BC}--\eqref{IC}. Moreover  $$\|\phi^n\|_{L^\infty}\le 1.$$
\end{theorem}
\begin{proof}
Assuming the sequence of the solutions of \eqref{AlgAC}--\eqref{AlgNSDivu} converge strongly in $ \HzVoltroisd \cap \HVoltroisd{2} \times L^2_0(\Omega)\cap H^1(\Omega) \times H^3(\Omega)$,
it is obvious that the limit of this sequence $(\Vec{u}^n,p^n,\phi^n)$ is a solution of \eqref{EqSysSD} with \eqref{BC}--\eqref{IC}. Moreover, under this assumption, using Theorem~\ref{Th-existnce} we have $\| \phi^n \|_{\Vec{L^{\infty}}(\Omega)} \le 1.$ 

To complete the proof all we need is to demonstrate the strong convergence of the solutions of \eqref{AlgAC}--\eqref{AlgNSDivu} in $ \HzVoltroisd \cap \HVoltroisd{2} \times L^2_0(\Omega)\cap H^1(\Omega) \times H^3(\Omega)$.
Subtracting \eqref{AlgAC} for two integers $k \neq m$ then using the test function $$\psi = - \Delta (\phi_{k+1} - \phi_{m+1})\in \Lund{2}$$ (from Theorem \ref{Th-existnce} we have $\phi_k \in \HVolund{3}$) 
and using \eqref{Cauchy_convect} once again
\begin{equation*}\begin{aligned}
\alpha_0  \| \nabla ( \phi_{k+1}  - &\phi_{m+1} ) \|^2 
+ \gamma \Delta t \| \psi \|^2 \\
& = \Delta t  \int_\Omega \left(
\Vec{u}_{k} \cdot \nabla (  \phi_{k+1} - \phi_{m+1}) 
+ 
( \Vec{u}_{k} - \Vec{u}_{m}  ) \cdot \nabla \phi_{m+1}
\right) \psi
\\
& \quad + \frac{3\gamma \Delta t}{\eta^2}  \int_\Omega
\left(\phi_k^2 ( \phi_{k+1} - \phi_{m+1})
+\phi_{m+1}( \phi_{k} + \phi_{m})  ( \phi_{k} - \phi_{m})\right)
\psi
\\
& \quad + \frac{\gamma \Delta t}{2\eta^2}(2\beta +1) 
\int_\Omega \left( \phi_{k} - \phi_{m}  \right)\psi
+ \frac{\gamma \Delta t}{2\eta^2} 
\int_\Omega \left( \phi_{k}^3 - \phi_{m}^3  \right)\psi
\end{aligned}\end{equation*}
which gives
\begin{equation*}\begin{aligned} 
\gamma \Delta t \| \Delta \left( \phi_{k+1} - \phi_{m+1}\right) \| 
& 
\le \Delta t  \left(  \Cu \| \nabla \left(  \phi_{k+1} - \phi_{m+1}  \right) \|  + \Cgfi \|  \left(   \Vec{u}_k - \Vec{u}_m \right)\| \right) 
\\
& \qquad +  \frac{3\gamma \Delta t}{\eta^2}\left(   
\| \phi_{k+1} - \phi_{m+1} \| + 
\left( 2\beta + 10 \right)  \| \phi_{k} - \phi_{m}   \|  \right) 
\end{aligned}\end{equation*}
From Corollary \ref{convergenceH1} 
we get the convergence of $\Delta \phi_{k}$ in $\Lund{2}$ hence, recalling Theorem~\ref{TheoBorneUPhi}, the convergence of $\phi_k$ in $\HVolund{2}$. 

Concerning the convergence  of $\phi_k$  in $\HVolund{3}$, we use the same technique. From Theorem~\ref{Th-existnce}, we can apply the $\nabla$ operator on \eqref{AlgAC} at $k$ and $m$. Subtracting both equations and using the test function 
$$\psi = \- \Delta \nabla (\phi_{k+1} - \phi_{m+1})\in \Lund{2}$$ 
we get
the convergence of $\Delta \nabla \phi_{k}$ in $\Lund{2}$ from which we conclude the convergence of $\phi_k$ in $H^{3}$.
For the convergence of the sequence $\Vec{u}_k$ in $\HVoltroisd{2}$, we 
introduce
\begin{equation*}\begin{aligned}
\Vec{v}_{k,m} &= \|\nabla \cdot \left( \mu_{k+1} \Vec{D} \left( \Vec{u}_{k+1} \right) -  \mu_{m+1}\Vec{D} \left( \Vec{u}_{m+1} \right) \right)\|\\
\Vec{v}^R_{k,m} &= \|\nabla\mu_{k+1}\Vec{D} \left( \Vec{u}_{k+1} - \Vec{u}_{m+1} \right) 
+ \delta\mu \nabla \cdot \left( (\phi_{k+1}-\phi_{m+1}) \Vec{D} \left( \Vec{u}_{m+1} \right)\right)\|
\\
\end{aligned}\end{equation*}
Based on  \eqref{AlgNSDivu} (Corollary \ref{convergenceH1} gives the convergence in $\HVoltroisd{1}$ for $(\Vec{u}_k)_{k \ge 0}$ and we prooved the convergence in  $\HVolund{3}$ for $(\phi_{k})_{k \ge 0}$), subtracting \eqref{AlgNSDivu} at $m\neq k$, we conclude that
$\Vec{v}_{k,m}$ and $\Vec{v}_{k,m}^R$ converge to zero. 
Then, from 
$$
\|\mu_{k+1}\nabla \cdot \Vec{D} \left( \Vec{u}_{k+1} - \Vec{u}_{m+1} \right)\| \le \Vec{v}_{k,m} + \Vec{v}^R_{k,m} 
$$
we get $(\Vec{u}_{k})_{k \ge 0}$ converges in $\HVoltroisd{2}$.
Finally, for the convergence of the pressure sequence, the usual argument based on the inf--sup condition is used (see~\cite{BoyFab2013} for example).
\end{proof}

This also completes the demonstration of Theorem~\ref{Le_resultat} which is merely a corollary of Theorem~\ref{convergenceH2}.

\section{Numerical results}\label{NumRes}
The goal of this section is 
to illustrate 
the effectiveness of a totally implicit scheme for the Navier-Sokes/Allen-Cahn model. 
%
With that in mind, a single numerical test is sufficient: an analytical test based on a two dimensional manufactured solution.
The fixed point strategy constructed to establish Theorem~\ref{Le_resultat}, offers a first numerical recipe to approximate the solution of \eqref{EqNS0}. 
\subsection{Methodology}
Obviously various strategies, efficient and well adapted to the physical context,  could be used to discretise and solve \eqref{AlgAC} and \eqref{AlgNSDivu}. Here, the finite element method has been retained for the spatial discretisation.

We used proper degrees of interpolation for each variables. For the fluid a Taylor-Hood interpolation was retained, for the velocity a quadratic interpolation (P2) and for the pressure a linear (P1) interpolation, insuring the respect of the inf--sup condition (see \cite{GueMinShe2006,BofBreFor2013}). As for the phase field we chose a P2 interpolation.
Since efficiency of the fully discrete algorithm is not the purpose of this work, we chose to leave the system as is: using a $LU$ solver on a \textit{pressure penalised} algebraic system for $\eqref{AlgNSDivu}$, with a value of $10^{-8}$ for the pressure penalisation.

For these tests, a uniform triangular mesh of sufficiently small size, $100\times 100$,  is used, this insure a negligible spatial error for the chosen range of time steps.
The algorithm is implemented using FreeFem++ a finite element software, see \cite{freefem}. 
 As for the fixed point (Step 5 in Algorithm~\ref{Algo1}), a tolerance of $10^{-9}$ on the sum of the norm of the variations of $\phi$ and $\Vec{u}$ is used as a stopping criteria.

\subsection{Algorithms}
Applying the spatial discretization  to \eqref{AlgAC}--\eqref{AlgNSDivu}, a basic strongly coupled implicit algorithm, called \textit{fully implicit Newton} method (FIN) is obtained 

\begin{algorithm}[H]
\caption{Simple totally implicit algorithm}\label{Algo1}
\begin{algorithmic}[1]
\STATE Given $(\Vec{u}^0, p^0, \phi^0)$ and $\Delta t$
\FOR{$n=1,...$}
\STATE $t^{n+1} = t^n+\Delta t$
\STATE $(\Vec{u}_0, p_0, \phi_0) = (\Vec{u}^n, p^n, \phi^n)$
\WHILE{not satisfactory} 
\STATE{Compute $\phi_{k+1}$ by solving the finite element version of \eqref{AlgAC}}
\STATE Update the density and viscosity with \eqref{AlgRhoMu}
\STATE{Compute $(\Vec{u}_{k+1}, p_{k+1})$ by solving the finite element version of \eqref{AlgNSDivu}}
\ENDWHILE
\STATE $(\Vec{u}^{n+1}, p^{n+1}, \phi^{n+1}) = (\Vec{u}_{k+1}, p_{k+1}, \phi_{k+1})$
\ENDFOR
\end{algorithmic}
\end{algorithm}
Since we are using a fixed point approach, we also consider the following linearization
$$
f(\phi_{k+1}) = \phi_{k+1}(\phi_{k+1}^2 -1) \approx \phi_{k+1}(\phi_{k}^2 -1).
$$
in \eqref{EqACDisc}. We call this modification a \textit{Picard fixed point} approach. In that case $\beta\ge~2$ suffice to insure uniform boundedness of the phase field and Theorem~\ref{Th-existnce} is valid with the modified bound on $\beta$. This simple linearisation seems a natural choice, we propose to use it as a comparative. Replacing \eqref{AlgAC} by 


\begin{equation}\label{AlgACPicard}
\left \{
\begin{array}{l}
\begin{aligned}
\left(1+ \frac{\gamma \Delta t}{\eta^2} \left( \beta + \phi_k^2 -1 \right)\right)\phi_{k+1}& + \Delta t \Vec{u}_k \cdot \nabla \phi_{k+1}
-\gamma \Delta t \Delta \phi_{k+1} \\
& = \phi^n + \frac{\gamma \Delta t}{\eta^2} \beta\phi_k \quad \mbox{in}\, \Omega
\end{aligned} \\
\partial_{\Vec{n}} \phi_{k+1} = 0\quad \mbox{on}\quad \partial \Omega
\end{array}
\right.
\end{equation}
we still get a strongly coupled fully implicit semi-discretization. Using the finite element method on \eqref{AlgACPicard}, \eqref{AlgRhoMu}--\eqref{AlgNSDivu} we get a new algorithm, called the \textit{fully implicit Picard} method (FIP).


As a second comparative approach, we built an \textit{explicit} formulation 
by putting $\beta=0$ and expliciting the non linear term. Replacing \eqref{AlgAC} by 
\begin{equation}\label{AlgACExplicit}
\left \{
\begin{array}{l}
\begin{aligned}
\phi_{k+1} + \Delta t \Vec{u}_k \cdot \nabla \phi_{k+1}
& -\gamma \Delta t \Delta \phi_{k+1}\\ & = \phi^n + \frac{\gamma \Delta t}{\eta^2} (1-(\phi^n)^2))\phi^n \quad \mbox{in}\, \Omega
\end{aligned} \\
\partial_{\Vec{n}} \phi_{k+1} = 0\quad \mbox{on}\quad \partial \Omega
\end{array}
\right.
\end{equation}
we get a \textit{strongly coupled explicit semi-discretization} approach. Using the finite element method on \eqref{AlgACExplicit}, \eqref{AlgRhoMu}--\eqref{AlgNSDivu} we get a new algorithm, called the \textit{strongly coupled explicit} method (SCE).

%
%
\subsection{Numerical test}
For the numerical experiments a problem having an analytical solution is used. It is inspired by the 
the finite element tests proposed in \cite[Accuracy test 4.2]{SheYan2010}
\begin{equation}
\begin{cases}
\Omega = ]-1,1[\times]-1,1[,\ \eta = 0.1,\ \gamma =1 ,\ \sigma = 1,&\\
T= 10\eta^2/(13\gamma),\ \rho_a = 3, \rho_b = 1, \mu_a = \mu_b = 1, & \\
\phi(t,x,y) = t(x+2)^2/(2T)-1, & \\
u_1(t,x,y) = \pi \sin(2\pi y)\sin^2(\pi x)\sin(t), & \\
u_2(t,x,y) = -\pi \sin(2\pi x)\sin^2(\pi y)\sin(t), & \\
p(t,x,y) = \cos(\pi x)\sin(\pi y)\sin(t).
\end{cases}
\end{equation}
Note that $\phi(t,x,y) \in [-1,1]$ is a polynomial of degree two in space and linear in time, therefore the approximation, using degree two in space and BDF1 in time, should be exact. Suitable boundary conditions on $\Vec{u}$ and $\phi$ and forcing terms depending on the analytical solution are added. 
%

Table~\ref{Le_tablo} gives the $L^2$ error on $\Vec{u}$ and $\phi$ at time $10/1\,300$. As a measure of precision we define
$$E_u = \max\limits_n \|\Vec{u}^n - \Vec{u}(t^n)\|,
\qquad E_\phi = \max\limits_n \|\phi^n - \phi(t^n)\|.$$
The computational cost is quantified using the total number of times the algebraic system corresponding to the Navier-Stokes equations \eqref{AlgNSDivu} is solved. We limit ourselves to a few time step, as it suffice to convincingly demonstrate our point.
For the fully implicit Newton method, the time step $\Delta t = \eta^2/(13\gamma)$ and parameter $\beta =9/8$ are fixed following Corollary~\ref{convergenceH1}.  
For the FIP method, we use the same time step but with $\beta=2$, a value insuring uniform boundedness of $\phi$ for this method. As the lower bound on $\beta$ is only a sufficient condition, two tests using $\beta = 0$ for the FIN and FIP methods are also presented. 

\begin{table}[t]
    \centering
    \begin{tabular}{l c c c c c c}
\hline
Method & $\beta$ & $\Delta t$ &N &\# of solves & $E_u$ & $E_\phi$ \\
\hline\hline
FIN       & 0 & 1/1\,300 & 10 & 40  & 6.74601e-07 & 9.89655e-11\\
FIN       & 9/8 & 1/1\,300 & 10 & 90  & 6.74601e-07 & 3.02437e-10\\
FIP       & 0 & 1/1\,300 & 10 & 97 & 6.74601e-07 & 2.47127e-10 \\
FIP       & 2 & 1/1\,300 & 10 & 102 & 6.74601e-07 & 3.32047e-10 \\
SCE  & 0 & 1/1\,300 & 10 & 31  & 0.00176404 & 0.031862 \\
SCE  & 0 & 1/13\,000 & 100 & 295 & 0.00020879 & 0.003257 
    \end{tabular}
    \caption{Precision ($L^2$ error) and computational cost (total number of algebraic system solved) with respect to the three methods (fully implicit Newton method (FIN), fully implicit Picard method (FIP) and strongly coupled explicit method (SCE)).}
    \label{Le_tablo}
\end{table}
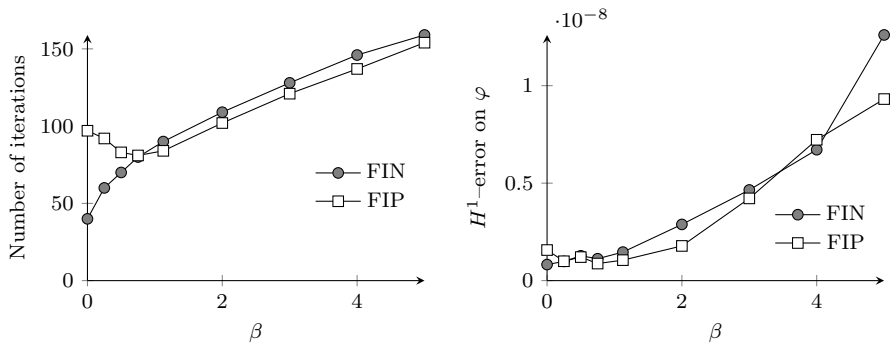
\begin{figure}[hbtp]
\footnotesize
\begin{center}
\begin{tikzpicture}[scale=1]
\begin{axis}[
width=.5\textwidth,
height=.4\textwidth,
xmax=5.,
xmin=0,
ymin=0,
xlabel={$\beta$},
ylabel={Number of iterations},
axis lines=left,
legend style={draw=none, at={(.65,.52)},anchor=north west},
    cycle list name=black white]
\addplot [mark=*, mark options={fill=black!50}]
coordinates {
( 0       ,     40     )
( 0.25    ,        60  )
( 0.5     ,       70   )
( 0.75    ,        80  )
( 1.125   ,         90 )
( 2       ,     109    )
( 3       ,    128     )
( 4       ,     146    )
( 5       ,   159      )
( 6       ,  177       )
( 7       ,   193      )
};
\addlegendentry{FIN}

\addplot  [mark=square*, mark options={fill=black!0}]
coordinates {
 (0     ,    97   )
 (0.25  ,       92)
 (0.5   ,      83 )
 (0.75  ,       81)
 (1.125 ,       84)
 (2     ,    102  )
 (3     ,    121  )
 (4     ,  137    )
 (5     ,    154  )
 (6     ,  171    )
 (7     ,  187    )
};
\addlegendentry{FIP}
\end{axis}[
\end{tikzpicture}
\quad
\begin{tikzpicture}[scale=1]
\begin{axis}[
width=.5\textwidth,
height=.4\textwidth,
xmax=5.,
xmin=0,
ymin=0,
xlabel={$\beta$},
ylabel={$H^1$--error on $\phi$},
axis lines=left,
legend style={draw=none, at={(.65,.35)},anchor=north west},
    cycle list name=black white]
\addplot [mark=*, mark options={fill=black!50}]
coordinates {
( 0       ,   8.23678e-10     ) 
( 0.25    ,      9.86187e-10  ) 
( 0.5     ,     1.27323e-09   ) 
( 0.75    ,      1.12277e-09  ) 
( 1.125   ,       1.46326e-09 ) 
( 2       ,   2.88311e-09     ) 
( 3       ,  4.65162e-09      ) 
( 4       ,   6.71539e-09     ) 
( 5       , 1.26021e-08       ) 
( 6       , 1.2288e-08        ) 
( 7       , 1.54155e-08       ) 
};
\addlegendentry{FIN}
\addplot  [mark=square*, mark options={fill=black!0}]
coordinates {
 (0     ,   1.56854e-09     ) 
 (0.25  ,      9.94185e-10  ) 
 (0.5   ,     1.20717e-09   ) 
 (0.75  ,      8.75327e-10  ) 
 (1.125 ,       1.05063e-09 ) 
 (2     ,   1.77934e-09     ) 
 (3     ,   4.21775e-09     ) 
 (4     , 7.22477e-09       ) 
 (5     ,   9.31331e-09     ) 
 (6     , 1.14541e-08       ) 
 (7     , 1.24367e-08       ) 
};
\addlegendentry{FIP}
\end{axis}[
\end{tikzpicture}
\caption{ Effect on $\beta$ on the FIN and FIP method in this specific case. On the left, the effect on the total number of iteration (computational efficiency). On the right, the $H^1$ norm of the error for both method. Values computed at $\beta = 0, 0.25, 0.5, 0.75, 1.125, 2 , 3, 4, 5$}
\label{La_fig}
\end{center}
\end{figure}
As usual in numerical methods the context of use is important. This example does not exclude the relevance of simpler approaches such as explicit or weakly coupled methods. It simply underline the effectiveness of implicit and strongly coupled methods. 
Concerning the use of a weakly coupled scheme, this relatively simple example does not seem to land itself to such methods. Weakly coupled  strategies were excluded from this analysis as in our very few experiment they all exhibited poor precision even when compared to the SCE method. Let summarize some observations resulting from Table~\ref{Le_tablo} and Figure~\ref{La_fig}
\begin{itemize}
    \item For the FIN and FIP method the level of precision for $\phi$ is such that we can consider the error on the velocity $\Vec{u}$ as produced by the finite element approximation of the Navier-Stokes equation.
    \item There is a clear gain in precision for the implicit methods, FIN and FIP, when compared to the explicit approach. At equal precision (or at equal time step), the choice of an implicit scheme is obvious. Based on the two experiments with the SCE method and the gap in precision, argument in favor of the explicit method, based on efficiency, cannot be made. 
    \item Generally speaking, considering the loss of information in the Picard method, a Newton-type approach seems preferable. However, in this specific case, from Figure~\ref{La_fig}, these two methods are relatively similar (in efficiency and precision) when using $\beta\in [0.75,4]$.
    \item The parameter $\beta$ is basically a "theoretical trick" producing Theorem~\ref{Le_resultat}. In the FIN method, $\beta>0$ can be interpreted as perturbing the Newton method. For the FIN method, from a numerical perspective, the use of $\beta=0$, is optimal in term of efficiency (Tableau~\ref{Le_tablo} and Figure~\ref{La_fig}). 

\end{itemize}




\section{Conclusion}

This work propose a proof of the well posed character of a fully implicit strongly coupled semi-discrete NS-AC system (Theorem~\ref{Le_resultat}). We consider this theoretical result as the basis needed for the analysis of various numerical scheme (not restricted to finite element approaches)) aiming at the approximation of the solution of the NS-AC system. The demonstration offered a simple and relatively efficient time scheme which allowed us to produce a finite element approximation. The numerical tests have established clearly the relevance of implicit and strongly coupled schemes for the NS-AS system, therefore reinforcing the significance of Theorem~\ref{Le_resultat}.


\section*{References}
\biboptions{sort&compress}
\bibliographystyle{model1-num-names}
\bibliography{banquebiblio}
\end{document}